\documentclass[12pt]{amsart}
\usepackage{latexsym, amsmath, amssymb,amsthm,amsfonts,amsopn,mathtools}
\usepackage{subfigure,comment,tikz}
\usetikzlibrary{calc}
\usepackage{tkz-euclide}
\usepackage{epsfig,graphics,color,graphicx,graphpap}
\usepackage{url}
\usepackage{hyperref}
\usepackage{cleveref}
\usepackage{bm}
\usepackage{amsmath}

\newcommand{\eqdef}{\overset{\text{\tiny def}}{=}}

\newcommand{\argmin}{\mathop{\mathrm{argmin}}}

\newcommand{\bb}{\mathbf{b}}
\newcommand{\bbz}{\mathbf{z}}
\newcommand{\be}{\mathbf{e}}
\newcommand{\dbb}{d\mathbf{b}}
\newcommand{\db}{db}
\newcommand{\bo}{\mathbf{\Psi}}
\newcommand{\rhobt}{\bm{\widetilde \rho}}
\newcommand{\rhob}{\bm{\rho}}
\newcommand{\s}{\mathfrak{s}}
\newcommand{\bxi}{\bm{\xi}}
\newcommand{\bpsi}{\bm{\psi}}

\newcommand{\SCap}{\mathrm{Cap}}
\newcommand{\Proj}{\mathrm{Proj}}

\usepackage{todonotes}
\newtheorem{proposition}{Proposition}[section]
\newtheorem{corollary}[proposition]{Corollary}

\newcommand{\abs}[1]{\lvert#1\rvert}

\newcommand{\bracket}[1]{\left[#1\right]}

\newcommand{\dt}{\partial_{t}}

\newcommand{\G}{\mathcal{G}}

\newcommand{\R}{\mathbb{R}}

\newcommand{\C}{\mathbb{C}}

\newcommand{\calH}{\nabla^2 H}

\newtheorem{thm}{Theorem}
\newtheorem{prop}{Proposition}[section]

\newtheorem{lem}[prop]{Lemma}

\newtheorem{rem}[prop]{Remark}

\renewcommand{\G}{\mathcal{G}}

\newcommand{\GG}{\underline{G}}
\newcommand{\ugamma}{\underline{\gamma}}
\newcommand{\K}{\mathcal{K}}
\renewcommand{\H}{\mathcal{H}}
\renewcommand{\Re}{\mathfrak{Re}}

\usepackage[foot]{amsaddr}

\title[Toy Model Bound States]{Discrete Bound States in a Toy Model for Weak Turbulence and Implications for the Invariant Measure}

\author[J.~L.~Marzuola]{Jeremy~L.~Marzuola$^1$} \address{$^1$Mathematics
  Department, University of North Carolina at Chapel Hill, Chapel
  Hill, NC 27599, USA} \email{marzuola@math.unc.edu}

\author[J. C. Mattingly]{Jonathan C. Mattingly$^2$}

\address{$^2$Mathematics Department and Department of Statistical Science, Duke University, Durham, NC 27708, USA}
\email{jonathan.mattingly@duke.edu}

\begin{document}

\begin{abstract}
  A model Hamiltonian dynamical system has been derived to study
  frequency cascades in the cubic defocusing nonlinear Schr\"odinger
  equation on the torus.  Here, we explore the framework for
  exploring a canonical ensemble formulation of the dynamics
  through classification of energy minimizers for fixed mass and characterizing
  the invariant measure in a neighborhood of those minimizers.
\end{abstract}

\maketitle

\section{Introduction}
\subsection{ The Toy Model}
\label{s:intro}
In \cite{CKSTT} the authors Colliander-Keel-Staffilani-Takaoka-Tao
studied the two-dimensional defocusing cubic toroidal nonlinear Schr\"odinger
equation,
\begin{equation}
  \label{e:dcnls}
  i u_t + \Delta u - |u|^2 u = 0, \ \ u(0,x) = u_0(x) \ \text{for} \ x \in \mathbb{T}^2,
\end{equation}
by developing their ``Toy Model System'' given by the equation
\begin{equation}
  \label{e:toy_model}
  -i\dt b_j (t) = -\abs{b_j(t)}^2 b_j(t) + 2 b_{j-1}^2 \overline{b_j}(t)
  + 2 b_{j+1}^2 \overline{b_j}(t),
\end{equation}
for $j \in \mathbb{Z}$. 
Solutions of this model can be considered on the entire lattice
or on a
finite lattice of the form $j=-N,\dots,0,\dots, N$ and
$\bb=(b_{-N},\dots, b_N) \in \C^{2N+1}$ with ``boundary
conditions''
\begin{equation}
  \label{e:dirichletbc}
  b_{-(N+1)}(t) = b_{N+1}(t) = 0.
\end{equation}
When considering the entire lattice, one can obtain well-posed
dynamics by assuming that $\bb \in \ell^2
(\mathbb{Z})$. Alternatively, one can consider solutions that are only
locally finite. While such ``infinite energy'' solutions are
interesting we will not consider them here.

The $b_j$ approximate the mass associated with families of resonantly
interacting frequencies in \eqref{e:dcnls}. Roughly the $b_j(t)$ are
the coefficients of a hierarchy of resonant spatial Fourier modes for a solution $u(x,t)$ to
\eqref{e:dcnls}. In \cite[Section 5]{CMOS1}, the authors derive a
discrete Burgers type equation with a phase drag term (see
\eqref{e:toy_model_hydro} below) and studied numerically the flux to
large modes of \eqref{e:toy_model}.
pi
In \cite[Section 3]{CKSTT}, the Toy Model \eqref{e:toy_model} is studied as a Hamiltonian dynamical system with Hamiltonian 
\begin{equation}
  \label{e:hamiltonian}
  H(\bb) = \sum_{j=-N}^N \left( \frac12 | b_j |^4 - 2 \Re ( \bar{b}_j^2 b_{j-1}^2) \right)
\end{equation}  
and  symplectic structure
\begin{equation}
  \label{e:symplectic}
  i \frac{db_j}{dt} = \frac{\partial H(\bb)}{\partial \bar b_j},
  \quad j\in \mathbb{N}.
\end{equation}
Here $\Re$ gives the real part of a complex number and $\bar b$
signifies the complex conjugate of $b$.
Formerly, the above dynamics makes sense on the bi-infinite lattice
obtained by setting $N=\infty$. Using the structure of the equations,
it can be seen for any $N>0$, including $N= \infty$, that if the
initial condition $b(0)$ is compactly supported (i.e.  supported on a
finite number of nodes), then the solution $b(t)$ remains compactly
supported for all time with a region of support that does not grow
with time. To keep our arguments simple, we will concentrate on
dynamics with $N < \infty$ while keeping in mind that this captures
the dynamics of initial conditions with finite support even when
$N= \infty$.

The model \eqref{e:toy_model} shares many of the symmetries of
\eqref{e:dcnls}, including phase invariance, scaling, time translation
and time reversal.  However, many of these symmetries are redundant in
the toy model \eqref{e:toy_model} in that the only  known invariant
beyond the Hamiltonian \eqref{e:hamiltonian} is the mass quantity
\begin{equation}
  \label{e:truncmass}
  M(\bb)=  \sum_{j=-N}^N |b_j|^2.
\end{equation}

As it was first shown in  \cite[Section 4]{CMOS1}, the dynamics of
\eqref{e:toy_model}, posses many fixed points. We will explain in
Section~\ref{Sec:Stationary} that there is an infinite family of fixed
points with compact support. The first result of this paper shows that
there is a distinguished family of fixed points, supported on $3$-modes and unique up to translations, that minimizes the Hamiltonian $H$ for a given $M$. More exactly, using discrete versions of the rearrangement theorems motivated by those in the calculus of variations often applied in the continuum setting, we establish the
following theorem.
\begin{thm}
  \label{mainthm}
  For a given mass  $m>0$, there exists a 3-node solution $\bb^*$ 
  that is a minimizer of $H$ with $M(\bb^*)=m$. It is the unique minimizer of $H$ up to translations and phase rotations.
  Namely for each lattice site $k \in \{-N+1,\dots,N-1\}$ and phase $\theta \in [0,2\pi)$, there exists a 
  minimizing $\bb^*=\bb^*(m,k,\theta)$ given by 
  $b_{k-1}^* = b_{k+1}^* = \sqrt{\frac{3m}{11}}e^{i \theta}$ and 
  $b_k^* = \sqrt{\frac{5m}{11}}e^{i \theta}$ and $b_j^*=0$ for all $j$ with $|k-j|>1$. This family represents all of the minimizers of the energy $H$ for fixed mass $m$.
\end{thm}

\begin{rem}\label{rem:b^* notation}
  For notational convenience we will write $\bb_k(m)$ for the fix point
  with $\bb^*(m,k,\theta)$ centered on the $k$th coordinate and with
  phase zero in the  $k$th coordinate. When we want to introduce a phase $\theta$ we will write
  $e^{i\theta} \bb_k(m)$. Lastly, we will denote by
  $B_k^*(m)$ the circle of fixed points $\{e^{i \theta} \bb_k(m) :
  \theta \in [0,2\pi]\}$ obtained by rotating $\bb_k(m)$ and
by  $B^*(m)=\bigcup_k B_k^*(m)$ the set of all energy minimizers with
mass $m$. When m = 1, we will drop the explicit dependence on m and simply write \( \bb_k \).
\end{rem}

We will see in the next section that there is a natural collection of statistically stationary solutions of \eqref{e:toy_model} corresponding to a family of Gibbs Measures at different temperatures. If one restricts these stationary solutions to those with finite mass $M$ but low temperatures, we would expect the solutions to concentrate around the 3-mode solution given in \Cref{mainthm}. It is then natural to ask what is the structure of the fluctuations around these minimizing 3-mode solutions. The main results of this paper show that to leading order these fluctuation are correlated in the 5-mode neighborhood of the  3-mode solution and independent white noise away from this region.

\subsection{Gibbs Measures}
To explore the statistical equilibrium of the toy model's dynamics, we begin by introducing a family of canonical Gibbs measures.  Defining circles $\mathbb{T}_j$ for
$j = 1,\dots,N$ as
\begin{multline}
  \mathbb{T}_j = \big\{ \bb = (b_{-N+1}, \dots, b_0, \dots, b_N) \ \big|  \  \\ |\bb|^2 = 1, \  |b_j| =1, \ b_k = 0 \ \text{for all} \ k \neq j \big\}, 
\end{multline}
the authors in \cite{CKSTT} point out that the flow of
\eqref{e:toy_model} leaves each
${\mathbb T}_j$ invariant.  In \cite{CMOS1}, it was observed that
\eqref{e:toy_model} also has a natural probabilistic formulation and
can be seen to have some basic recurrence properties.  There, the
authors considered the white noise measure
\begin{equation}
  \label{e:massGibbs}
  \begin{aligned}
     \frac1{Z_N} \exp\big(- \tfrac12 M(\bb)\big)\dbb
                     \qquad\text{where}\qquad   \dbb=\prod_{j=-N}^N \db_j\,,
  \end{aligned}
\end{equation}
$b_j = \alpha_j + i \beta_j$, $\db_j= d \alpha_j d \beta_j$, and  some
normalizing constant $Z_N$.\footnote{We hope that the use of $\beta$
  as the inverse temperature and $\bb_j=\alpha_j+\beta_j$ will not
  cause any confusion as the meaning should be clear from context.}
Since $M(\bb)$ is a non-negative quadratic function this Gibbs measure is a Gaussian measure.

Since the Hamiltonian $H$ is also conserved, it is tempting to consider an invariant Gibbs measure $\mu^N_{\beta}$ for the $N$-mode toy model system with 
\begin{equation}
  \label{invmeas_mu}
  \mu^N_{\beta}(\dbb) \propto e^{-\beta H(b)} \dbb.
\end{equation}
From the perspective of statistical mechanics, this would correspond
to the canonical ensemble with temperature $\frac1\beta$.

Unfortunately, while this measure is well defined locally, it cannot
be normalized to obtain a probability measure since  $\{H(\bb) : \bb
\in \C^{2N+1}\} = (-\infty,\infty)$. To see that $\H$ is not bounded
from below, and hence $\mu^N_{\beta}$ is not normalizable, we argue as follows.

In
\Cref{prop:H_bound} below, we will see that if we restrict the sphere of unit mass
$S(1) = \{ \bb \in \C^{2N+1} : M(\bb) = 1 \}$  we have $H(\bb) \geq
-\tfrac{7}{22}$ were this lower bound is realized. Hence by continuity,
there exist a open set $A_1 \subset  S(1) $ of positive Lebesgue measure so $H(\bb) \leq -\tfrac{1}{4}$
for all $\bb \in A_1$. Since $H(\lambda \bb) = |\lambda|^4 H(\bb)$, we
have that the maximum of $H$ over $A_\lambda =\{ \lambda \bb : \bb \in
A_1\}$ is bounded from above by $-\tfrac{\lambda^4}{4}$. By
increasing $\lambda$ we can make the value of $H$ on $A_\lambda$ be
come uniformly, arbitrarily negative. This construction also hints at
the fact that we can control $H$ as long as we restrict to a set where
$M$ is bounded. After generalizing the questions slightly, we will
explore this direction below.

More generally for any $F\colon \R^2\rightarrow \R$, we can define
\begin{equation}
  \label{invmeas}
  \nu^N_{\beta,F}(\dbb) \propto e^{-\beta F(M(\bb),H(\bb))} \dbb
\end{equation}
and ask what conditions on $F$ ensure that this measure be normalized to form an invariant probability measure. 

Notice that unlike the measure in \eqref{e:massGibbs}, this is not necessarily a Gaussian measure.

Our path to understanding this question begins with a simple
observation, $H(\bb)$ is bounded from above and below on $S(m)$ for
each fixed $m \geq 0$. More exactly, we have the following
result.

\begin{proposition}\label{prop:H_bound}
  For any $N >0$, we have
  \begin{align}
  \label{eq:enbds}
    -\tfrac{7}{22} M(\bb)^2 \leq H(\bb)\leq \tfrac{3}{4} M(\bb)^2 .
  \end{align}
  These bounds are sharp.
\end{proposition}
\begin{proof}
  The lower bound in \eqref{eq:enbds} follows from Theorem \ref{mainthm} (see \eqref{eqn:Hminval}).  The upper bound in \eqref{eq:enbds} it turns out arises from a two mode solutions of the form $(\sqrt{M}/2,-\sqrt{M}/2)$, but it plays no real role in what follows here as in this result we focus on the lower bound in relation to the concentration of the Gibbs measure.
\end{proof}

\begin{rem}
    While we are concerned with the minimum energy solution, in a forthcoming statistical mechanics/random dynamics study \cite{panos2025stats}, the authors observe interesting dynamics in a region of energy and mass where the long time behavior appears to be dominated by the energy {\it maximizers} of our problem.  Our rearrangement methods used in the proof of Theorem \ref{mainthm} should also be applicable to assess the properties and classification of such states in a similar manner, but we leave such a study to future work.
\end{rem}

From this we immediately obtain the following sufficient condition for $\nu^N_{\beta,F}$ to be well defined.
\begin{corollary}Defining $F^*(m)=\max(F(m,C_Nm^2), F(m,-c_Nm^2)$ if
  \begin{align*}
    \int e^{-\beta F^*(M(\bb))}\dbb < \infty  
  \end{align*}
  then $\nu^N_{\beta,F}(\C^{2N+1}) < \infty$ and thus can be normalized to be an invariant probability measure.
\end{corollary}
In light of \Cref{prop:H_bound}, we see that $|H|$ is bounded on any
set where $M$ is bounded (which is related to the fact that if a sequence is in $\ell^2$, then it is $\ell^4$). This motivates the consideration of the
measure $\mu^N_{\beta}(\dbb)$ conditioned to lie on the sphere of
radius $\sqrt{m}$ given by $S(m)=\{
\bb \in \C^{2N+1} : M(\bb)=m\}$.

We will focus here on what can be thought of as the {\it canonical Gibbs measure}.  Denote by $\Gamma_m(\dbb)$ the Hausdorff measure
on $\C^{2N+1}$, viewed as $\R^{4N+2}$, restricted to $S(m)$ and
normalized so that $\Gamma_m(S(m))= m^{2N +\frac12}\s_{4N+2}$ where
$\s_{4N+2}= \frac{2\pi^{2N+1}}{\mathrm{Gamma}(2N+1)}$ surface area of
a real unit sphere in $4N+2$ dimensions and $\mathrm{Gamma}$ the standard Gamma Function. 

Since for $m >0$, $S(m)$ is smooth Riemannian manifold, the  volume measure on the surface  on  $S(m)$ normalized to have total measure $\s_{4N+2}$ coincides Hausdorff measure  $\Gamma_m(\dbb)$ coincides with the standard Riemannian volume measure on $S(m)$. We have chosen the normalization so that $\Gamma_m(\dbb)$ coincides with the restriction of the standard Lebesgue measure from $\C^{2N+1}\sim\R^{4N+2}$ to $S(m)$ or equivalently the Riemannian surface measure on $S(m)$ inherited by restricting the standard metric on $\C^{2N+1}$ to $S(m)$. In particular,  we have the following disintegration result for any $f\colon \C^{2N+1} \rightarrow \R$
\begin{align*}
  \int_{\C^{2N+1}} f(\bb)\dbb &= \int_0^\infty\Big(\int_{S(r^2)} f(\bb) \Gamma_{r^2}(\dbb)\Big) dr \\ &=\int_0^\infty \frac{1}{2 \sqrt{m} }\Big(\int_{S(m)} f(\bb) \Gamma_m(\dbb)\Big) dm .
\end{align*}
The first equality is just Fubini's Theorem  while the second equality shows the effect of changing the variable of integration from $r$ to $m=r^2$.

Applying this representation to the measure $\mu^N_{\beta}(\dbb)$ from \eqref{invmeas_mu} with the assumption that $\int f(\bb) \mu^N_{\beta}(\dbb)< \infty $, we see that
\begin{align*}
  \int_{\C^{2N+1}} f(\bb)  \mu^N_{\beta}(\dbb)&\propto\int_0^\infty \frac{1}{2 \sqrt{m} }\Big(\int_{S(m)} f(\bb) \mu^N_{\beta,m}(\dbb)\Big) dm ,
\end{align*}
where
\begin{align*}
   \mu^N_{\beta,m}(\dbb)\eqdef e^{-\beta H(\bb)}\Gamma_m(\dbb).
\end{align*}
From this representation it is clear that it is sufficient to
understand the measures $ \mu^N_{\beta,m}(\dbb)= e^{-\beta
  H(\bb)}\Gamma_m(\dbb)$ to understand $\mu^N_{\beta}(\dbb)= e^{-\beta
  H(\bb)}\dbb$ from \eqref{invmeas}. The conditioned measures
$\mu^N_{\beta,m}(\dbb)$ have the advantage that they can always be
normalized to a probability measure. The measure $\mu^N_{\beta,m}$ gives
the canonical ensemble at temperature $\frac1\beta$ and fixed mass $m$.

A similar representation can be written for  $\nu^N_{\beta,F}$ and used to better understand its structure. In the interest of simplicity, we will concentrate on   $\mu^N_{\beta}(\dbb)$, though analogous calculations can be made for $\nu^N_{\beta,F}$. We will be particularly interested in the structure of $ \mu^N_{\beta,m}$ as $\beta \rightarrow \infty$ which corresponds to the low temperature limit.

\begin{rem}
  The marginal measure of $\mu^N_{\beta}$ in the mass $m$ is always well defined and can be written as
  $\bar\mu^N_{\beta}([a,b])= \mu^N_{\beta}\big(\bigcup_{m \in[a,b]}
  S(m)\big)$ for $ 0\leq a \leq b< \infty$. Then we have the
  disintegration $\mu^N_{\beta}(\dbb)= \mu^N_{\beta,m}(\dbb)
  \bar\mu^N_{\beta}(dm)$ and we see that $\bar\mu^N_{\beta}(dm)
  \propto  \frac{1 }{2 \sqrt{m} }dm$.
\end{rem}

For notational convenience for any test function $\phi\colon \C^{2N+1}
\rightarrow \R$, we define 
\begin{align}\label{mu-beta-m}
  \mu_{\beta,m}^N\phi\eqdef 
\int_{S(m)} \phi(\bb) e^{-
   \beta  H(\bb)  } \Gamma_m(\dbb). 
\end{align}
We then have to following result which will be implied by the more
general result given in  \Cref{prop:avgGeneral} from \Cref{sec:rotinv}.

\begin{proposition}\label{lem:avg} For any $\theta$ and $\bb \in \C^{2N+1}$, 
  $H(e^{i\theta}\bb) =  H(\bb)$. As a consequence 
  for any test function $\phi\colon \C^{2N+1}
\rightarrow \R$, $ \mu_{\beta,m}^N\phi =  \mu_{\beta,m}^N\overline \phi$
where 
\begin{align}\label{eq:average}
   \overline \phi (\bb) = \frac1{2\pi} \int_0 ^{2\pi} \phi(e^{i\theta}\bb) 
  d \theta .
\end{align}
\end{proposition}

\subsection{Low temperature behavior}

We are interested in the "low temperature
limit" of the measure as $\beta \rightarrow \infty$ while the mass $M(\bb)$ is
constrained to be fixed.  On the entire lattice, since $H(\bb)< \infty$
implies $M(\bb)<\infty$, but not the converse, constraining $M$ does not
limit the value of $H$. However, the measures given in \eqref{invmeas}
preferentially weights states with lower $H$. Hence, it would not be
surprising that when restricting to a shell of constant $M$ as
$\beta \rightarrow \infty$, $\mu_N^{\beta}$ concentrates around the
states $\bb$, which minimize $H(\bb)$ for constant mass
$M$.

\begin{thm}
\label{thm:Gibbs}  For any smooth test function $\phi\colon \C^{2N+1}
\rightarrow \R$, we have that
  \begin{align*}
  \frac{e^{\beta h_*(m)}}{\beta^N}  \mu_{\beta,m}^N\phi  \xlongrightarrow{\beta\rightarrow\infty}
 \sum_{k=-N+1}^{N-1}
    \int_0^{2\pi}\phi(e^{i\theta}\bb_k^*)d\theta .
  \end{align*}
  In particular,
  \begin{align*}
     \frac{ \mu_{\beta,m}^N}{\mu_{\beta,m}^N(S(m))}  
  \end{align*} 
is a probability measure that  converges to the uniform measure on the set
of three mode solutions with mass $m$  that minimizer the energy $H$.  
\end{thm}

Using the same framework, we will establish the following characterization
of the invariant measure for this problem, stated here as a
meta-theorem.  A more precise statement will be given in Section
\ref{Sec:Meas}.

We begin by recalling from \Cref{rem:b^* notation}, that $B^*(m)$ is the
set of all minimizers with mass $m$ and observe that this a compact
set in $\C^{2N+1}$ consisting of the $2N-1$ sets $B_k^*$ that are
disjoint circles. The distance between the set $B^*(m)$ and any point $\bb$ is well
defined for all $\bb$. This minimal distance is achieved at a unique
point, denoted $\widehat \bb^*(\bb)$, except for an exceptional set made
of planes in $\C^{2N+1}$ of at most  ${4N+1}$ complex
dimensions.\footnote{Indeed, if a vector $\bb$ is equidistant from two elements of our set, it means that there exists some complex vector $\bbz$ such that $\Re \langle \bbz, \bb \rangle = 0.$}

\begin{rem}\label{rem:non-unique_distance}
Observe that this exceptional set as well as its 
intersections with the sphere $S(m)$ are of zero Lebesgue measure in 
$\C^{2N+1}$ and $S(m)$. Hence this non-uniqueness can either be 
ignored as we will only be integrating the function  $\widehat
\bb_*$ against  Lebesgue measure in 
$\C^{2N+1}$ and $S(m)$ or it can be ``defined away'' by choosing a
measurable selection for the non-unique points. 
\end{rem}
Thus with the caveats in \Cref{rem:non-unique_distance}, we define
\begin{align}
  \label{eq:hat_b*}
  \widehat \bb^*(\bb) = \argmin_{\bb^* \in B^*} |\bb - \bb^*|,
\end{align}
and consider the orthogonal projection away from $\widehat \bb^*(\bb)$
given by 
\[
\Proj_{\bb^*}^\perp(\bb)=\bb - \Proj_{\bb^*}^\perp(\bb)
\]
where
\begin{align}
  \label{eq:hat_b*proj}
  \Proj_{\bb^*}(\bb)=\big \langle \bb,  \tfrac{\widehat \bb^*(\bb)}{|\widehat \bb^*(\bb)|} \big \rangle \tfrac{\widehat \bb^*(\bb)}{|\widehat \bb^*(\bb)|}.
\end{align}
We can then define the function $G\colon \C^{2N+1} \rightarrow [0,\infty)$ by
\begin{align}\label{eq:G}
  G(\bb)\eqdef \frac12 \big\langle \big(\tfrac{14}{11}M(\bb) I
  +\nabla^2 H(\widehat \bb^*(\bb))\big) \Proj_{\bb^*}^\perp(\bb),  \Proj_{\bb^*}^\perp(\bb) \big\rangle.
\end{align}
Note, here and throughout we take the $\langle {\bf u} , {\bf v} \rangle = \Re ({\bf u} \cdot {\bf v}) $ to be the inner-product on the full $4N+2$ dimensional space of real vectors such that ${\bf u}$ and $i {\bf u}$ are orthogonal to each other.

We now use the  function $G$ to specify a ``Gaussian-like'' measure
$\gamma_{\beta,m}^N$ on $\C^{2N+1}$, defined by
\begin{align}\label{eq:gaussian}
  \gamma_{\beta,m}^N \phi \eqdef \int_{S(m)} \phi(\bb) e^{-\beta G(\bb)}\Gamma_m(\dbb)
\end{align}
for bounded functions $\phi\colon S(m) \rightarrow \R$. The measure
$\gamma_{\beta,m}^N$ capture the deviation from the set $B^*$ in the ambient
space $\C^{2N+1}$ but wrapped around the sphere $S(m)$ in the spirit
of the von Mises–Fisher distribution.  This will become clearer when we analyze the measure further in Section \ref{Sec:Meas}.

In defining $G$ we could have equally used the intrinsic geodesic
distance on the sphere $S(m)$ to define the measure. As we will see in
\Cref{sec:intrinsicG}, this will produce a slightly different
Gaussian-like measure $\ugamma_{\beta,m}$ defined using a function
$\GG$ defined using the intrinsic distance. We will see that this
measure, any many other related formulations are all equivalent as
$\beta \rightarrow \infty$. We postpone discussion of this alternative
measure to \Cref{sec:intrinsicG} and \Cref{lem:intrisicGamma}.

With the definition of $G$ given, we now state an informal version of a result
capturing the fluctuations around the limiting result given in
\Cref{thm:Gibbs}.

\begin{thm}\label{thm:LLN}
  Fix and $m >0$. For any smooth $\phi \colon S(m) \rightarrow \R$ , we have that
  \begin{align*}
  \frac{e^{\beta h^*(m)}}{\beta^{N}}  \mu_{\beta,m}^N \phi 
    \xrightarrow{\beta \rightarrow \infty}   \frac1{\beta^{N}}  \gamma_{\beta,m}^N\phi
  \end{align*}
  and the expression on right-hand side is uniformly bound from above
  and below as $\beta \rightarrow \infty$. In particular,
  \begin{align*}
     \frac1{\beta^{N}}  \gamma_{\beta,m}^N\phi 
 \xrightarrow{\beta \rightarrow \infty} 
     \sum_{k=-N+1}^{N-1}
    \int_0^{2\pi}\phi(e^{i\theta}\bb_k^*)d\theta,
  \end{align*}
  where the right-hand side should be compared with the expression in \Cref{thm:Gibbs}.
\end{thm}
\Cref{thm:LLN} will follow from a more precise version \Cref{thm:Gibbs} stated as \Cref{thm:Gibbs_precise} in Section \ref{Sec:Meas}, combined with some observations that occur in
\Cref{rem:contration}, \Cref{c:GareTheSame}, and \Cref{prop:avgGeneral}.

Since it plays a major role in the analysis below, we record also some properties of the operator, $ \tfrac{14}{11}M(\bb) I
  +\nabla^2 H(\widehat \bb^*\big)$ for a given optimizer $\bb^*\in B^*$.  We will state the result for the real optimizer centered at lattice site $j=0$, denoted $\bb^*_0$, but by phase and translation invariance similar results hold for any $\bb^*$ as described in \Cref{mainthm}.  We will use the notation of $\be_j$ for standard basis vectors such that $\be_j$ has $e_j = 1$ and $e_k = 0$ for all $k \neq 0$.

  \begin{thm}
      \label{thm:Hessprops}
     Let $\bb_0^*$ be the real optimizer of mass $M(\bb^*_0) = m$ from Theorem \ref{mainthm} centered at lattice site $0$.  The operator $\nabla^2 H(\bb^*_0)(j,k)$ 
     only has non-zero entries when $j,k \in \{ -2,-1,0,1,2 \}$.  The spectrum of 
     \begin{equation*}
            \tfrac{14}{11} m I + \nabla^2 H(\bb^*_0)
     \end{equation*}
       is as follows:
      \begin{enumerate}
          \item  There is $1$ negative eigenvalue given by $-m\tfrac{28}{11}$ with corresponding eigenvector $\bb^*_0$ itself;
          \item  There is a $1$ dimensional null-space spanned by $i \bb^*_0$;
          \item The remaining eigenvalues are positive and consist of
       the eigenvalues $m \times \tfrac{14}{11}$ with eigenvectors $\be_j$ for $j \notin \{-2,-1,0,1,2\}$, as well as the eigenvalues
 $m \times \left(\tfrac{2}{11},\tfrac{2}{11} ,\tfrac{26}{11},\tfrac{26}{11}\right)$
 with eigenvectors
 \begin{align*}
   \be_{-2},\be_2 , i \be_{-2}, i \be_{2},
 \end{align*} 
 eigenvalues  $m \times \left(\tfrac{12}{11},\tfrac{40}{11} \right)$ with eigenvectors 
  \begin{align*}
 \be_1-\be_{-1}, i (\be_1-\be_{-1}),
 \end{align*}
 and eigenvalues
 $m \times \left(\tfrac{60}{11},8 \right)$ 
 with the remaining eigenvectors being given by
 \begin{align*}
\sqrt{\tfrac53} (\be_1 + \be_{-1}) - 2 \be_0, i  (\sqrt{\tfrac53} (\be_1 + \be_{-1}) - 2 \be_0) 
 \end{align*}
 respectively.  
\end{enumerate}
  \end{thm}

  We thus have that for any $\bxi$ such that $\langle \bxi, \bb^*_0\rangle = \langle \bxi, i \bb^*_0\rangle$, that 
\begin{equation}
\label{eqn:GopBound}
    \langle (\tfrac{14}{11} m I + \nabla^2 H(\bb^*_0) )\bxi, \bxi \rangle \geq \frac{2m}{11} | \bxi|^2 .
    \end{equation}
Also, one observes that $\nabla^2 H (\bb^*_0) \bb^*_0=
-m\tfrac{42}{11}\bb^*_0$. The proof of \Cref{thm:Hessprops} will be
given by direct construction in \Cref{sec:Hessian}.

\begin{rem}
  The approximation gives some interesting insight into the structure of samples from $\mu_{\beta,m}$ as $\beta \rightarrow \infty$. Such as sample looks primarily like one of the three-mode minimizers described in \Cref{mainthm}, whose phase and location of its maximum are chosen uniformly at random over the set of possibilities. The fluctuations around this maximum are Gaussian with a correlated structure over the five-modes centered on the maximum and uncorrelated white noise away from these center five-modes.
\end{rem}

\subsection{Connections to Previous Work and outline of the result}

The invariance and characterization of a Gibbs measure has been
studied in the setting of dispersive PDEs such as the KdV equation, the
nonlinear Schrödinger equation and others since the work of
\cite{lebowitz1988statistical}, see for instance the developments
\cite{bourgain1996invariant,bourgain1997invariant,bringmann2024gibbs,bringmann2024invariant,oh2018pedestrian,tolomeo2023phase,thomann2010gibbs,
  burq2018remarks} only to name a few. The computations in our problem
are simplified by the fact that it is from the start discrete with
dynamics covered by a simple ODE system with the optimizers of the
Hamiltonian localized to a finite box.  This allows us to easily see
invariance of the Gibbs measure of our dynamics, as well as to work on
measures over finite boxes to simplify and highlight the concentration
of measure phenomenon in the neighborhood of an energy minimizer.

For other works related to frequency cascades and the study of weak
turbulence for NLS, we refer the reader to
\cite{bourgain1995cauchy,bourgain1995aspects,bourgain1996growth,B04}
as well as the interesting and recent works \cite{Carles:2012jv,
  Hani2,GK12, Ku1,Soh1,GT1, grebert2012beating, guardia2012growth,
  haus2012dynamics, faou2013weakly,hani2013modified}.  See also the
work \cite{herr2016discrete} for studies of dynamics and fluxes in
\eqref{e:toy_model}.

\subsection{Organization of Paper}
In Section
\ref{Sec:Stationary}, we recall the family of stationary solutions to
\eqref{e:toy_model} as found in \cite{CMOS1}, including a $3$-mode
solution we explicitly describe that turns out to be the optimizer described in \Cref{mainthm}.  In Section \ref{s:enmin}, we
introduce a novel approach to discrete rearrangement arguments and
prove that this compact three mode solution is the energy minimizer
for fixed mass up to symmetries of the equation, establishing \Cref{mainthm}.  Then, in Section
\ref{Sec:JacHess}, we look at linearizing about the $3$-mode solution,
computing the gradient and Hessian of the energy $H$ at the $3$-mode
solution. In \Cref{Sec:Gaussian}, we describe collection ``Gaussian''
measure that are equivalent as $\beta \rightarrow \infty$. Using the
information on the Hessian $\nabla^2H$ from \Cref{Sec:JacHess} we
explore their structure. This allows us to describe the measure in a neighborhood of
minimizer with extreme precision in Section \ref{Sec:Meas}.  Finally,
we discuss Corollaries of our findings and future directions in Section
\ref{discussion}.

\section*{Acknowledgements}
The majority of this article was written in collaboration with Dana Mendelson.  It would not have come to fruition without her many insights and ideas.  J.L.M.\ acknowledges support from the NSF through grant DMS-2307384
and FRG grant DMS-2152289.  J.C.M. acknowledges support of the NSF
through the grants DMS-1613337 and RTG-DMS-2038056. All of the authors
thank the hospitality of the NSF Mathematical Sciences Research
Institute (MSRI, now SLMath) for its support and hospitality during the 2015
program semester long ``New Challenges in PDE: Deterministic Dynamics
and Randomness in High and Infinite Dimensional Systems'' of which all
the authors were members and during which this work was started.

\section{Exact in phase solutions}
\label{Sec:Stationary}

In this section, in order to motivate the three-mode optimizers described in Theorem \ref{mainthm}, we divert ourselves briefly to consider a notion of stationary solution for the dynamical model \eqref{e:toy_model}.  Given the structure of the equation, it is natural to consider exact solutions of fixed mass to \eqref{e:toy_model} that will be of the form 
\begin{equation}
\label{eqn:statansatz}
    \bb = e^{i \omega t} \bb_0,
\end{equation}
where $\bb_0$ is independent of time.  While these solutions do vary in time, their mass and Hamiltonian energy do not of course.  Such states were first introduced in \cite[Section 4]{CMOS1}.  In many similar equations, these stationary solutions serve as reasonable candidates for energy minimizers with fixed mass as they arise naturally from the corresponding Euler-Lagrange equations associated to $H$ and $M$. See the recent article \cite{parker2024standing} for an in depth study of related questions, as well as \cite{germain2017compactons} for a treatment of related solutions in a continuum version of our problem.  

Note, since we will prove that stationary solutions of this type are indeed energy minimizers for fixed mass, the fact that we have rotation invariance in our measure in \Cref{thm:Gibbs} is essential.  Indeed, our description of the measure gives a sense of how solutions that start near this orbit must stay near this orbit over all time, which is related to orbital stability of the model.

In what turns out to greatly simplify our particular framework, we begin by performing the Madelung transformation of \eqref{e:toy_model} by setting
\begin{equation}
  b_j(t) = \sqrt{\rho_j(t)}\exp(i\phi_j(t))
\end{equation}
with $\rho_j \geq 0$ and $\phi_j \in \mathbb{R}$ to obtain the evolution
equations:
\begin{equation}\label{e:toy_model_hydro}
  \begin{aligned}
    \dot\phi_j & = -\rho_j + 2 \rho_{j-1} \cos\bracket{2(\phi_{j-1}-\phi_j)} + 2 \rho_{j+1} \cos\bracket{2(\phi_{j+1}-\phi_j)} , \\
    \dot\rho_j & = -4 \rho_j \rho_{j-1}
    \sin\bracket{2(\phi_{j-1}-\phi_j)} -4 \rho_j \rho_{j+1}
    \sin\bracket{2(\phi_{j+1}-\phi_j)}.
  \end{aligned}
\end{equation}
This formulation makes it clear that phase interactions play a key
role in the dynamics. One class of interesting solutions stay in phase for all time, maintaining 
\begin{equation}
  \label{e:inphase}
  \phi_j(t) = \phi_{j+1}(t).
\end{equation}
Such solutions are said to be {\it phase locked}.  Assuming this
holds, \eqref{e:toy_model_hydro} becomes
\begin{subequations}
  \begin{align}
    \dot\phi_j & = -\rho_j + 2 \rho_{j-1} + 2 \rho_{j+1} , \\
    \dot\rho_j & = 0.
  \end{align}
\end{subequations} 
Hence, we observe that the Madelung transformation has changed our nonlinear stationary state problem into a constrained linear one since we require $\rho_j \geq 0$ for all $j$.  

Given our ansatz \eqref{eqn:statansatz}, we then require that  
\begin{equation}\label{eq:discreteLap}
  -\rho_j + 2 \rho_{j-1} + 2 \rho_{j+1}=\omega \in \R, \quad\text{for
    $j=1,\ldots N$},
\end{equation}
where $\omega$ is independent of $t$ and $\rho_0 = \rho_{N+1} = 0$ then the resulting solution will be phase locked 
with time independent amplitudes $\rho_j$. The phases would evolved
according to $\phi_j(t) = \phi_j(0) + \omega t$. One can also understand
the group rotational velocity of $\omega$ as the Lagrange Multiplier associated to the constraint that $M$ is fixed.  

The entire dynamics of \eqref{e:toy_model} then simply rigidly rotates with a group rotational velocity of $\omega$. 
The condition in \eqref{eq:discreteLap} corresponds to the linear system
\begin{equation}
  \label{e:upmat}
  \begin{pmatrix}  
    -1 & 2 & 0 &\cdots & 0 \\
    2 & -1 & 2 & \cdots &0 \\
    \vdots & \vdots & \ddots &\ddots & \vdots\\
    0 & 0 & \cdots & 2 & -1
  \end{pmatrix} \boldsymbol{\rho} = \omega \boldsymbol{1}.
\end{equation}
Though the matrix is tri-diagonal, it is not diagonally dominant, so
its solvability is not immediately clear.  However,
\begin{thm}[Thm 4.1 and Cor. 4.1 in \cite{CMOS1}]
  The matrix in \eqref{e:upmat} has no kernel for any $N$.    As a result, any nontrivial phase locked solution has $\omega \neq 0$.
\end{thm}

For a given system size $N$, the linear system can always be solved, and a phase
matched solution of \eqref{e:toy_model_hydro} exists.\footnote{
   This result first appeared in \cite{CMOS1}, where we point out that we can in fact find a solution for any real valued function $\omega(t)$.}  However, this
will not always yield a solution of \eqref{e:toy_model}.  As Figure
\ref{f:compact_solns1}$(d)$ shows, at $N=5$, the solution is not strictly
positive, and the Madelung transformation cannot be inverted.  Despite
the obstacle at $N=5$, we can again obtain a strictly positive
solution at $N=8$ and higher, as Figure \ref{f:compact_solns2} shows.
In \cite{CMOS1}, it is pointed out that a positive solution exists for
a special sequence of $N_k \to \infty$ as $k \to \infty$ using some
recursive linear algebra techniques.  However, while this family of solutions was discovered there, how these solutions compared in terms of the energy \eqref{e:hamiltonian} was not explored.

\begin{figure}
  \subfigure[$N=2$]{\includegraphics[width=1.0in]{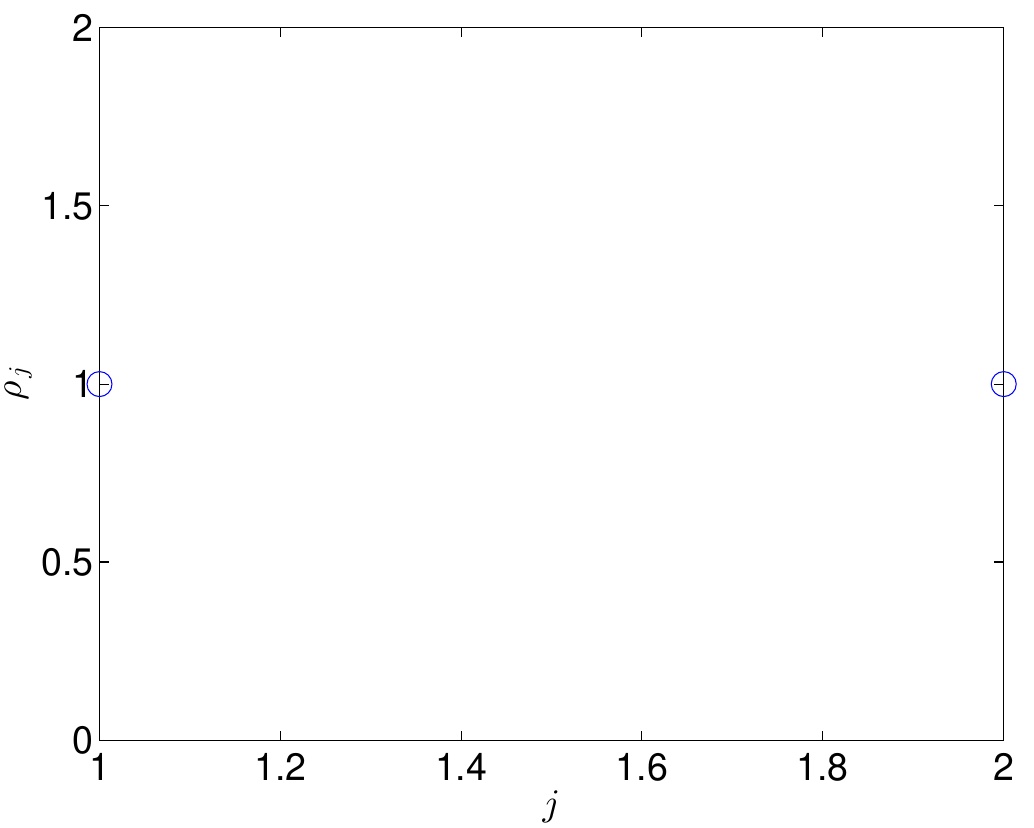}}
  \subfigure[$N=3$]{\includegraphics[width=1.0in]{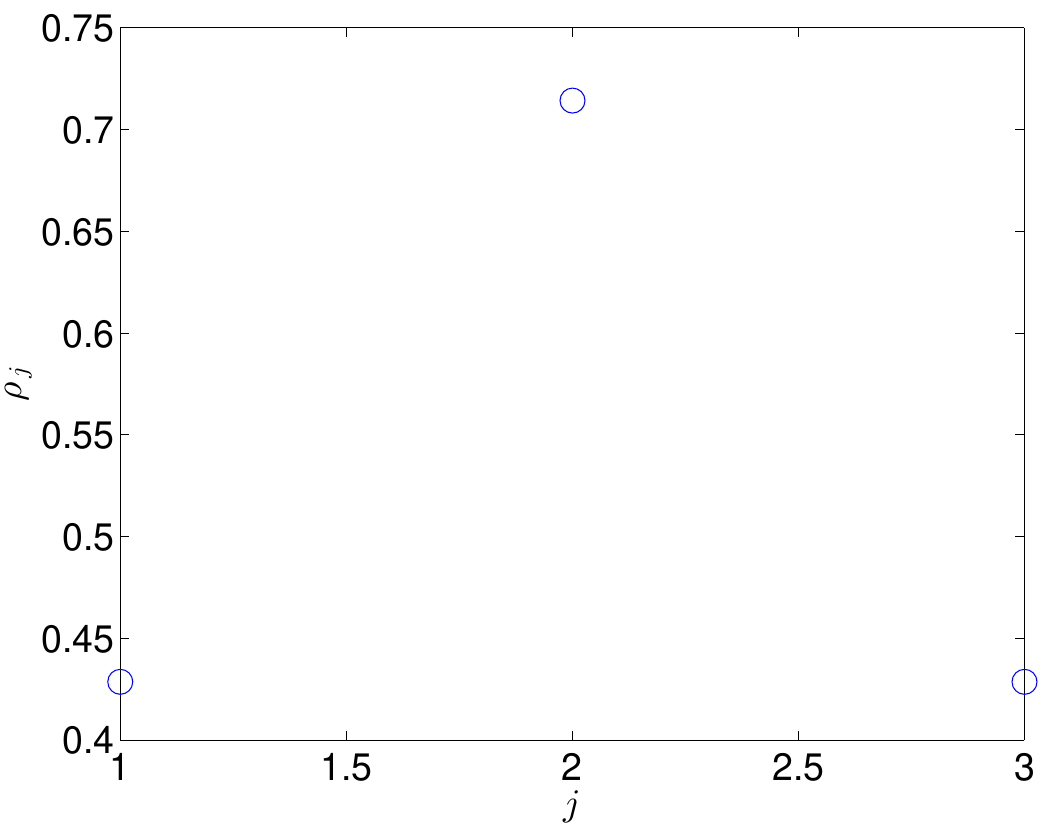}}
  \subfigure[$N=4$]{\includegraphics[width=1.0in]{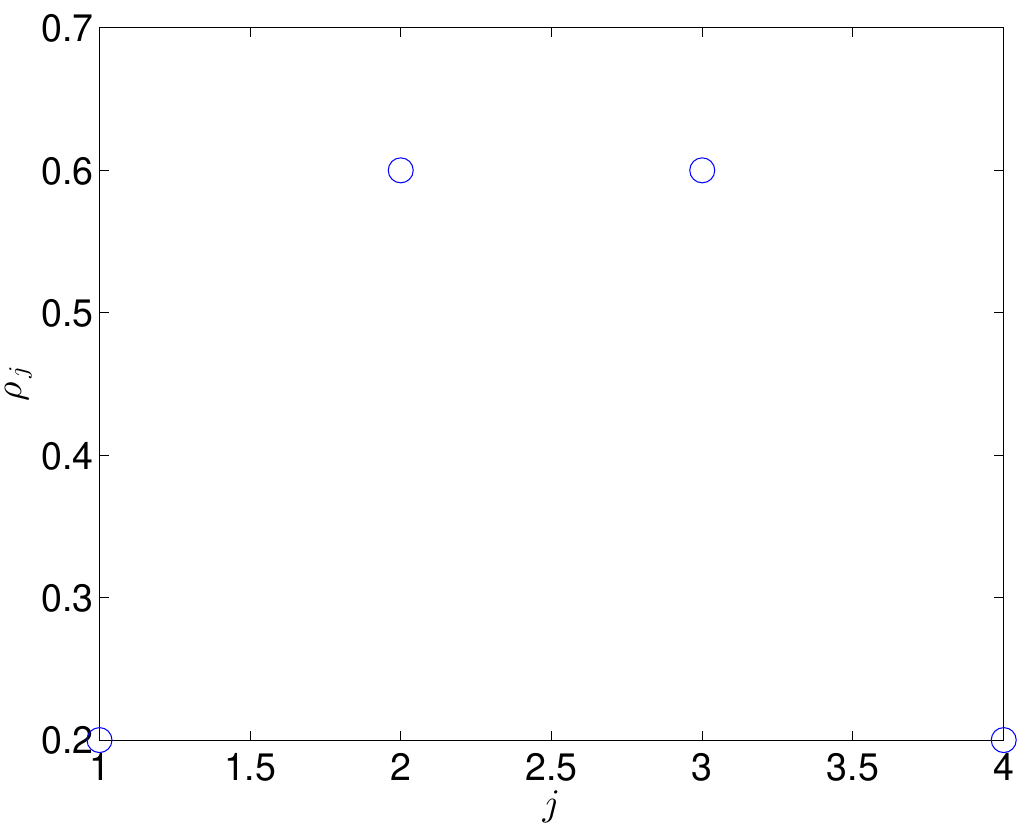}}
  \subfigure[$N=5$]{\includegraphics[width=1.0in]{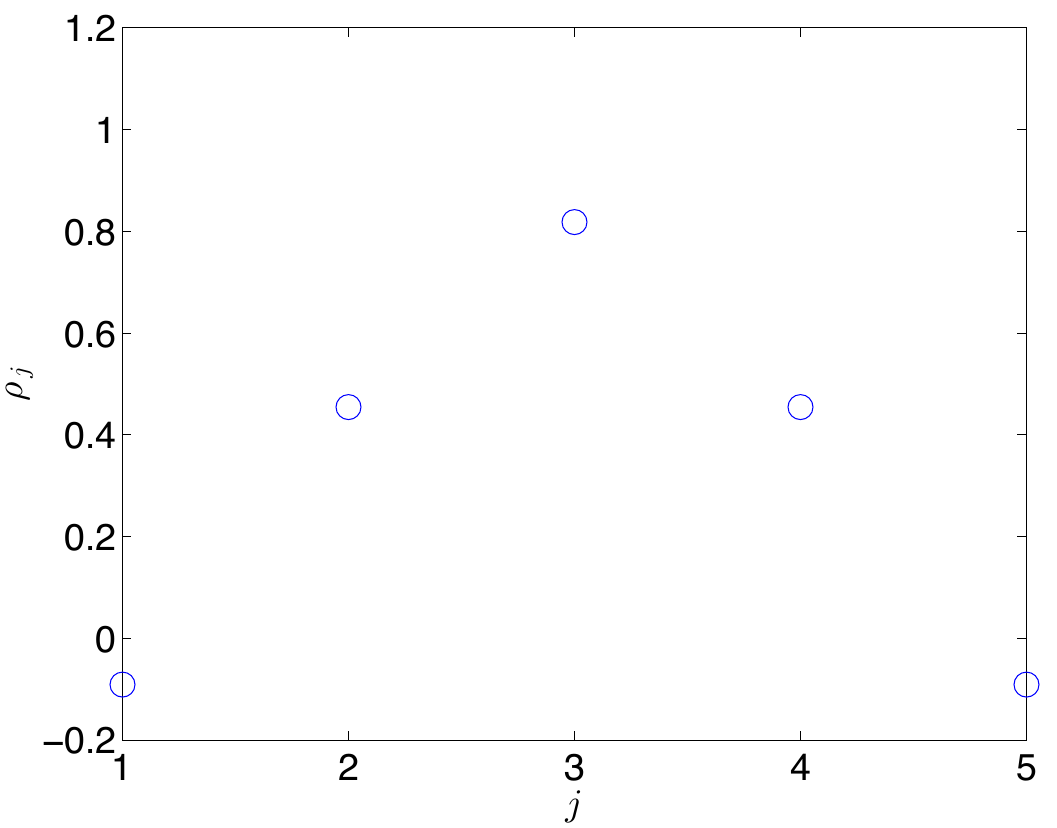}}
  \caption{Solutions of \eqref{e:upmat} with $\omega =1$.  The
    solution at $N=5$ is not strictly positive.}
  \label{f:compact_solns1}
\end{figure}

\begin{figure}
  \subfigure[$N=8$]{\includegraphics[width=1.4in]{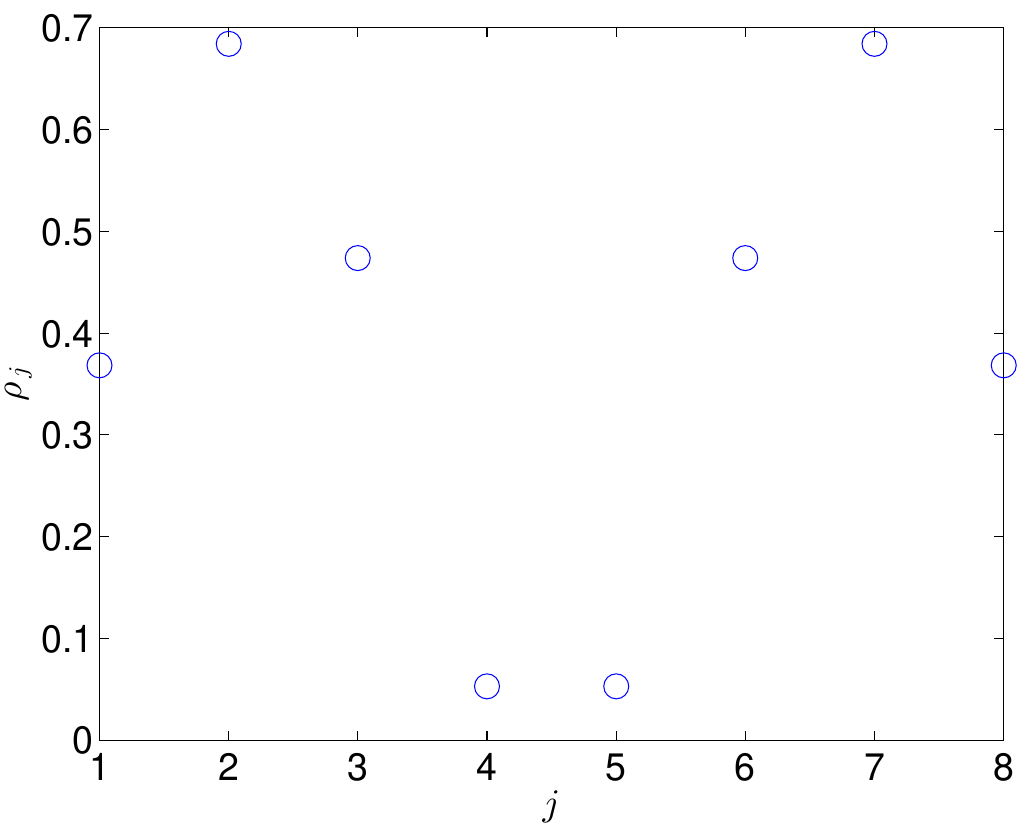}}
  \subfigure[$N=142$]{\includegraphics[width=1.4in]{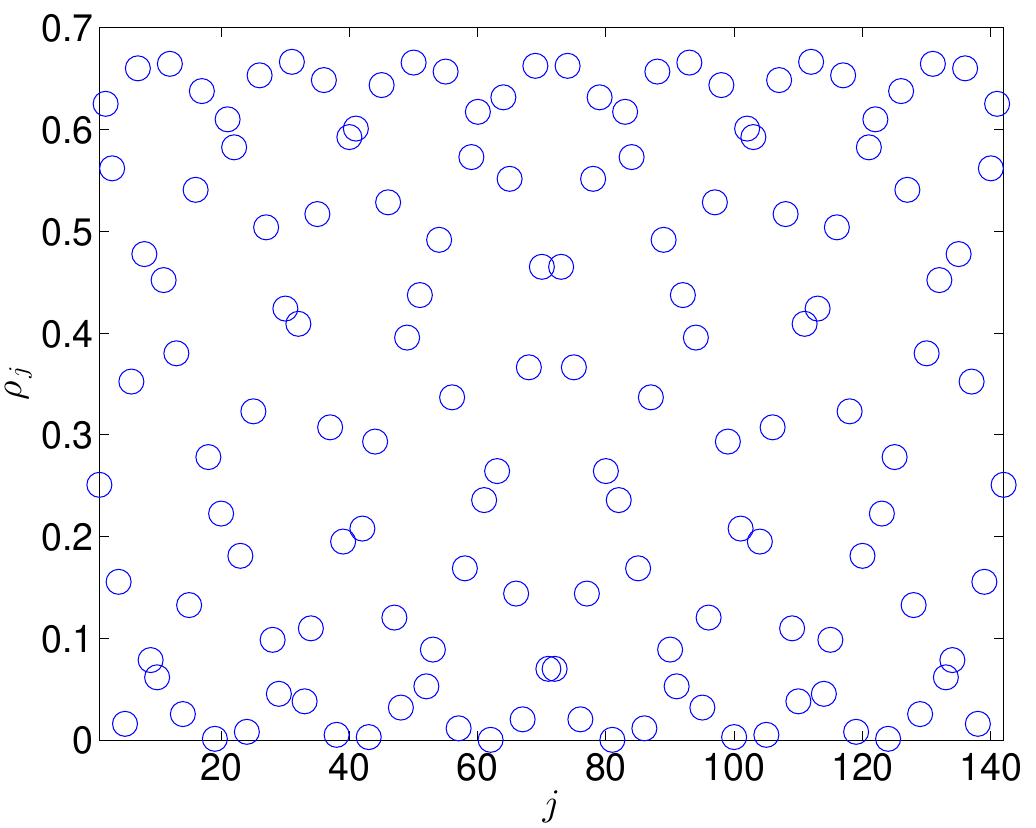}}

  \caption{Solutions of \eqref{e:upmat} with $\omega =1$. Both
    solutions are strictly positive.}
  \label{f:compact_solns2}
\end{figure}

Note that since the solutions we have constructed are all supported on a compact number of nodes and the equations for the toy model, the hydrodynamic formulation
and \eqref{e:upmat} are autonomous in $j$, one can easily concatenate these
localized solutions together to form new solutions. 

\subsection{The 3 mode solution}\label{three_mode}
\label{Sec:minimizer}

We will prove Theorem \ref{mainthm} in Section \ref{s:enmin} below that the 3 mode in-phase
solution described above is actually the global minimizer.

We begin this exploration by better understanding the relationship
between the phase $\omega$ and the mass $m$ in a solution concentrated
on the three modes $\{-1,0,1\}$.  We thus seek a solution of the form
\begin{align*}
& \rho_{-1} + \rho_0 + \rho_1= m, \\
 -&\rho_{-1} + 2 \rho_0= \omega, \\
 -&\rho_{0} +2 \rho_{-1} + 2 \rho_1 = \omega, \\
 -&\rho_{1} + 2 \rho_0 = \omega, 
\end{align*}
which can be further simplified by using the symmetry
\begin{equation}
\rho_{-1} = \rho_1 .
\end{equation}
The solution of the resulting system of equations is given by
\begin{equation}
\label{enmin}
\rho_{-1} = \rho_1 = \tfrac{3 m}{11}, \ \ \rho_0 = \tfrac{5 m}{11}, \ \ \omega = \tfrac{7 m}{11}.
\end{equation}
We can linearize about this simple three mode solutions, study flux
dynamics in given topologies, and ask about nearby dynamics.  Linear
stability of such solutions was considered in
\cite{parker2024standing}.

\section{Energy Minimizers for Fixed Mass}
\label{s:enmin}

We will demonstrate that the three mode solution from \eqref{enmin} is
in fact the energy minimizer for fixed mass on the entire lattice
$\mathbb{Z}$.  As a result, the three mode solution from \eqref{enmin}
will be the energy minimizers on any finite lattice for any
$N \geq 3$.

To demonstrate this fact, we begin by comparing the  symmetric, thee-mode solution found given
in \eqref{enmin} to the other symmetric, energy minimizing  solutions
with mass $m$ with one, two and four nodes. First observe that the 
energy of the symmetric, thee-mode solution is 
\begin{equation}
    \label{eqn:Hminval}
    H(\rho) = -\tfrac{7}{22} m^2. 
\end{equation}
One should compare this to the $1$-mode solution, $\rhob = (0,m,0)$
with energy $\frac12 m^2$, the $2$-mode solution, $\rhob = (0,
m/2,m/2,0)$, with energy $-\frac14 m^2$, and the $4$-mode solution,
$\rhob = (0,m/8,3m/8,3m/8,m/8,0)$ with energy $-\frac{5}{16} m^2$.
Since $-\tfrac{7}{22} <- \tfrac{5}{16} < -\tfrac14 <  \tfrac12 $, the symmetric 
$3$-mode solution  is the minimizer among this collection special
solutions.  While this is just an indication that our claim is
correct, it will be useful in the more general proof which we now turn
to.

The proof that the three mode solution from \eqref{enmin}  is the
global minimizer will require a number of preliminary results and
calculations. 

Using the $(\rhob,\bm{\theta})$ variables introduced in
the previous section, we observe that the Hamiltonian can be re-written as
\begin{equation}
  \label{e:hamiltonian_rhotheta}
  H(\rhob, \bm{\theta}) = \sum_{j=-\infty}^\infty \Big( \frac12 \rho_j^2 - 2  \rho_j \rho_{j-1} \cos\big( 2 (\theta_j - \theta_{j-1})\big) \Big)\,
\end{equation}
using $\rho_j = |b_j|^2$ and $b_j=\sqrt{\rho_j} e^{i\theta_j}$.

\begin{proposition}
  \label{prop:inPhaseMin}
  The minimizers of $H(\mathbf{\rho}, \mathbf{\theta})$ must be an {\it in phase} solution, that is a solution where $\theta_j = \theta_{j-1}$. In particular, $\min_{(\mathbf{\rho},\mathbf{\theta})} H(\mathbf{\rho},\mathbf{\theta})  =  \min_{\mathbf{\rho}} H(\mathbf{\rho})$ where
 \begin{align}
\label{altHinphase}
H( \rhob) = \sum_{j=-\infty}^\infty  \tfrac12 \rho_j^2 - 2 \rho_j \rho_{j-1}\,.
\end{align}
\end{proposition}
\begin{proof}
The first conclusion follows from the observation that $\cos( 2 (\theta_j - \theta_{j-1}))$ ranges over $[-1,1]$ and is maximized when $\theta_j = \theta_{j-1}$. Since the $\rho_j \geq0$, this corresponds to minimum of $H(\mathbf{\rho},\mathbf{\theta})$ when $\mathbf{\rho}$ is fixed and $\mathbf{\theta}$ is left to vary. The second observation follows from the fact that $\frac12 \rho_j^2 - 2  \rho_j \rho_{j-1} \cos( 2 (\theta_j - \theta_{j-1})) = \frac12 \rho_j^2 - 2  \rho_j \rho_{j-1}.
$ when $\theta_j = \theta_{j-1}$.  
\end{proof}

The next proposition implies that we need only consider $\mathbf{\rho}$ with $m=\sum \hat\rho_j = 1$ when exploring the structure of the minimizing $\mathbf{\rho}$.
\begin{proposition}
Let $\mathbf{\rho}$ be a minimizer with mass $1$, that is
\begin{equation*}
  H(\rhob) = \min_{\bm{\hat \rho}} \{ H(\bm{\hat \rho}) | \sum \hat\rho_j = 1 \}.
\end{equation*}

Then $m \rhob$ is the minimizer with mass $m$ for any $m \geq 0$.
\end{proposition}

\begin{proof}
  The proof follows from the fact that $\sum m\rho_j= m \sum \rho_j$
  and $H(m\rhob)= m^2 H(\rhob)$. This implies that $\min_{\bm{\hat \rho}} \{ H(\bm{\hat \rho}) | \sum \hat\rho_j = m \} =  \min_{\mathbf{\hat \rho}} \{ H(m\bm{\hat \rho}) | \sum \hat\rho_j = 1 \}=m^2 \min_{\bm{\hat \rho}} \{ H(\bm{\hat \rho}) | \sum \hat\rho_j = 1 \}$.
\end{proof}

Our next goal is to show that the three mode solution, discussed in Section \ref{three_mode}, is the unique minimizer for $H(\mathbf{\rho})$ in \eqref{altHinphase}. We will first demonstrate that a minimizer is non-increasing from its maximum value.

\begin{lem}\label{lem:dec}
Any minimizer of the energy $H$  with finite is always non-increasing about its maximum value.
\end{lem}

\begin{proof}
  Without loss of generality, we assume that the mass is $1$.  Since
  $\sum \rho_j =1$, we know that $\rho_j \rightarrow
  0$ as $j \to \pm \infty$. This implies that we may find a unique (up to multiplicity)
  maximum value and that the multiplicity of this maximum is finite.

  The proof will proceed by assuming that if $\rho$ is not
  non-increasing about a maximum value, one can construct a new
   $\rhobt$ with $H(\rhobt)< H(\rhob)$. The requisite $\rhobt$
  will be a rearrangement of $\rhob$. Since rearrangements do not
  alter the first expression in \eqref{altHinphase}, it suffices to
  show that we will decrease the second term in
  \eqref{altHinphase}. For future reference, we define the function
  $K$ by $H(\rhob)=\big(\sum \rho_j^2 \big)- 2 K(\rhob)$. To
  summarize, our goal is to define a rearrangement $\rhobt$ so that
  $K(\rhobt) > K(\rhob)$ as this will imply $H(\rhobt) <H(\rhob)$.

  We will explain how to construct the rearrangement on $\{\rho_j : j >0\}$ as a
  symmetric argument works for $j < 0$. We will return to this at the end of the proof. We will produce $\rhobt$ through an integrative induction procedure. $\rhob^{(n)}$ will represent the  $n$th step in the induction.

   We assume that the current $\rhob^{(n)}$ satisfies five properties:
   \begin{enumerate}
     \item  $\rhob^{(n)}$ is a rearrangement of the positive indices of $\rhob$ and agrees with $\rhob$ on the non-positive indices.
     \item The elements $(\rho_0^{(n)},\dots,\rho_{n}^{(n)})$ are monotone non-increasing.
     \item $\rho_{n}^{(n)} \geq \rho_{i}^{(n)}$ for all $i > n$.
     \item $K(\rhob^{(k)}) < K(\rhob^{(k+1)})$ for all $0 \leq  k < n$.
     \item $\rho_j^{(k)}= \rho_j^{(k+1)}$ for all $0\leq j \leq k$ and $0 \leq  k < n$.
  \end{enumerate}
  If we set $\rhob^{(0)}=\rhob$, then the induction hypothesis is satisfied at $n=0$ since it is clearly a rearrangement, $\rho_0 \geq \rho_j$ for all $j >0$, and the last two conditions are trivially satisfied as there are no $k$ satisfying  $k \geq 0$ and  $k <0$.

  If we can proceed inductively, we will succeed in creating a sequence $\rhob^{(n)}$ of successive rearrangements of $\rhob$ so that the $n$th step is non-increasing between $0$ and $n$ and
     $K(\rhob)=K(\rhob^{(0)}) > K(\rhob^{(1)})> \cdots >  K(\rhob^{(n)})>  K(\rhob^{(n+1)})$.

  Furthermore since if $i >n$, $\rhob^{(n)}$ and $\rhob^{(i)}$ agree in the first $n$ coordinates, the limit $\rhob^{(\infty)}$ is well defined and is a non-increasing rearrangement of $\rho$ with $ K(\rhob) >  K(\rhob^{(\infty)})$. Setting $\rhobt=\rhob^{(\infty)}$  completes the proof.

  All that remains is to show that if the induction hypothesis is satisfied at level $n$, we will construct a $\rhob^{(n+1)}$ so it is satisfied at level $n+1$. There are three cases with the last case having two sub-cases.

  \begin{itemize}
  \item[Case 1]:  If
$\rho_{n+1}^{(n)} \geq \rho_j^{(n)}$ for all $j >n+1$, then define
$\rhob^{(n+1)}=\rhob^{(n)}$. This choice satisfies the induction hypothesis at level $n+1$.  Clearly $\rhob^{(n+1)}$ is a rearrangement of $\rhob$ since $\rhob^{(n)}$ was. The second two induction conditions are satisfied since  $(\rho_{0}^{(n+1)},\dots,\rho_{n+1}^{(n+1)})$ is non-increasing  by the induction hypothesis and  $\rho_{n+1}^{(n+1)} \geq \rho_j^{(n+1)}$ for all $j > n+1$ by the assumption of this case. The last two induction assumptions hold since $K(\rhob^{(n+1)})=K(\rhob^{(n)})$ and $\rhob^{(n+1)}_i=\rhob^{(n)}_i$ for all $i$.
\item[Case 2:]If there is a $k > n$ such that $\rho_k^{(n)} > \rho_n^{(n)} $,  we will choose a $k$
  such that   $\rho_k^{(n)} \geq  \rho_j^{(n)} $ for all $j \geq n$. This is always possible since $\rho_j^{(n)} \rightarrow 0$ as $j \rightarrow \infty$. We now split this case into two sub-cases. In both cases we define
  $\rho_{n+1}^{(n+1)}=\rho_{k}^{(n)}$,  $\rho_{j}^{(n+1)}=\rho_{j-2}^{(n)}$ for $j \in \{n+3,\dots k\}$, and
  $\rho_{j}^{(n+1)} =\rho_{j}^{(n)}$ for $j \geq k+2$ or
  $0\leq j \leq n$. This leaves $\rho_{n+2}^{(n+1)}$ and $\rho_{k+2}^{(n+1)}$ unassigned and  $\rho_{k-1}^{(n)}$ and  $\rho_{k+1}^{(n)}$ utilized. We make different choices for these values depending on the relative magnitude of $\rho_{k+1}^{(n)}$ and  $\rho_{k-1}^{(n)}$.

  In both cases, it will be clear by the construction that $\rhobt^{(n+1)}$ will be a rearrangement and that $\rho_{j}^{(n+1)}=\rho_{j}^{(n)}$ for $0\leq j\leq n$. From the induction hypothesis and the fact that $\rho_{n+1}^{(n+1)}=\rho_k^{(n)} \geq  \rho_j^{(n)} $ for all $j \geq n$, we see that $(\rho_{0}^{(n+1)},\dots,\rho_{n+1}^{(n+1)})$ is non-increasing and that $\rho_{n+1}^{(n+1)} \geq \rho_{j}^{(n+1)}$ for $j \geq n+2$. All that remains is to make the remaining two assignments and show that $K(\rhob^{(n+1)}) > K(\rhob^{(n)})$. This last requirement forces that following two cases.
  \begin{itemize}
  \item[Case 2a:]If $\rho_{k+1}^{(n)} \geq \rho_{k-1}^{(n)}$, we set
    $\rho_{n+2}^{(n+1)} = \rho_{k+1}^{(n)}$ and 
    $\rho_{k+1}^{(n+1)} = \rho_{k-1}^{(n)}$. To check the final induction hypothesis, observe that
    \begin{multline*}
        K(\rhob^{(n+1)}) -  K(\rhob^{(n)})  = \rho_n (\rho_k - \rho_{n+1}) + \rho_{k+1} (\rho_{n+1} - \rho_{k+2})\\ + \rho_{k-1} (\rho_{k+2}- \rho_k) \\
                = (\rho_n - \rho_{k+1}) (\rho_k - \rho_{n+1} ) + (\rho_{k-1} - \rho_{k+1})( \rho_{k+2} - \rho_k) > 0.
    \end{multline*}
\item[Case 2b:] If $\rho_{k+1}^{(n)} < \rho_{k-1}^{(n)}$, we set
   $\rho_{n+2}^{(n+1)} = \rho_{k-1}^{(n)}$ and 
   $\rho_{k+1}^{(n+1)} = \rho_{k+1}^{(n)}$. To check the final induction hypothesis, observe that
   \begin{multline*}
     K(\rhob^{(n+1)}) -  K(\rhob^{(n)})   = \rho_n (\rho_k - \rho_{n+1}) + \rho_{k-1} (\rho_{n+1} - \rho_{k-2})\\
     + \rho_{k+1} (\rho_{k-2}- \rho_k) \\
 = (\rho_n - \rho_{k-1}) (\rho_k - \rho_{n+1} ) + (\rho_{k+1} - \rho_{k-1})( \rho_{k-2} - \rho_k) > 0.
   \end{multline*}
  \end{itemize}
\end{itemize}
The above procedure rearranged the positive indices. The negative indices can be dealt with in a completely analogous way. Alternatively, one can reflect the $\rhob^{(\infty)}$ obtained above about $j=0$ and rerun the above algorithm on the resulting state to obtain a $\rhobt$ that is monotone non-increasing in both directions.
 \end{proof}
 \begin{rem}
  The rearrangement procedure in the proof takes an initial $\rhob$ with infinitely many non-zero elements and produces a strictly positive non-increasing rearrangement $\rhob^{(\infty)}$. If the initial $\rhob$ has only finitely many non-zero elements then the resulting  $\rhob^{(\infty)}$ is compactly supported with a strictly positive connected region and infinitely many zeros on either side.
 \end{rem}

\begin{lem}
\label{lem:5over3}
 Any  minimizer $\rhob$  of the energy $H$  with finite mass has $\rho_0 \leq \frac56 (\rho_1 + \rho_{-1})$.  
\end{lem}

\begin{proof}
  We will show that if $\rho_0 > \frac56 (\rho_1 + \rho_{-1})$, then we can rearrange and decrease the energy.  Consider a solution of the form
  \[
  \rhob = (\ldots, \rho_{-2}, \rho_{-1}, \rho_0, \rho_1, \rho_2,
  \dots).
  \]
  We consider mass conserving rearrangements of the form
  \begin{equation*}
    \rhobt = (\ldots, \rho_{-2}, \rho_{-1}+\frac{\epsilon}{2}, \rho_0-\epsilon, \rho_1+\frac{\epsilon}{2}, \rho_2, \dots)
  \end{equation*}
for $\epsilon > 0$ sufficiently small.  Note the $\frac12$ on the
$\rho_1$ term that occurs in order to conserve the symmetric mass
constraint.  Then, the inequality follows by a simple check as to when
$H( \rhobt) < H( \rhob)$.  Indeed, calculating we have
\begin{align*}
H(\rhobt) & = H(\rhob) + \epsilon( \rho_{-1} +\rho_1 -2 \rho_0) + \frac32 \epsilon^2 \\
& \hspace{1cm} - 4 (  \frac{\epsilon}{2} ( \rho_{-2}  + 2 \rho_0 - 2(\rho_1 + \rho_{-1})  + \rho_{2} )  -  \epsilon^2 ) \\
& = H(\rhob) + \epsilon\big( 5 (\rho_1 +\rho_{-1} )  - 6 \rho_0   -2 ( \rho_{-2}  + \rho_{2}   )   \big)  +  \frac{11}{2} \epsilon^2.
\end{align*}
Hence, if $ 5 (\rho_1 +\rho_{-1} ) - 6 \rho_0 < 0$ (or equivalently 
$\rho_0 > \frac56 (\rho_1 +\rho_{-1} ) $), then we can decrease the
energy by choosing $\epsilon$ sufficiently small.  
\end{proof}

\begin{rem}
Note, the $3$ node solution given in \eqref{enmin} satisfies that
$\rho_0 = \frac56 (\rho_{-1} + \rho_1)$ and thus saturates the
estimate in Lemma \ref{lem:5over3}.
\end{rem}

Recall that the support of $\rho$ is the smallest interval $I \subset \mathbb{Z}$ so that $\rho_j=0$ if $j \not \in I$. As usually, if the set of such intervals is empty the support is taken to be $(-\infty,\infty)$. We now address the support of the minimizer.
\begin{lem}
\label{lem:compact}
Any minimizer of the energy $H$ with finite mass is compactly supported and is strictly positive on its supported.
\end{lem}

\begin{proof}
  We prove the second claim first. By \Cref{lem:dec}, we know that any finite mass energy minimizer must be non-increasing around its maximum values. Let $\rho_j=0$ be a location in the support of a minimizer $\rhob$. Without loss of generality, let us assume that $j$ is to the right of the   maximum values. Since the $\rhob$ is non-increasing, $\rho_k=-$ for all $k \geq j$ which contradicts the fact that $j$ was in the support to start with.

  We now turn to the first claim that any such minimizer has compact
  support.  Again, with out loss of generality we can assume that it
  has mass one.  By Lemma \ref{lem:dec}, let us assume without loss of
  generality that the maximum of the minimizer exists at $j=0$ and
  that $\rho$ is non-increasing off of the maximum. Suppose there
  exists a minimizer $\rhob = \{ \rho_j \}_{j=-\infty}^\infty$ with mass
  $1$ and $\rho_j > 0$, for all $j \in \mathbb{N}$.  By Lemma
  \ref{lem:5over3}, we may further assume
  \begin{equation*}
\rho_0 \leq \frac56 (\rho_1 + \rho_{-1}),    
  \end{equation*}
although in fact we only need
\begin{equation*}
  \rho_0 < 2 (\rho_1 + \rho_{-1}).
\end{equation*}
The finite mass condition implies that for any $\epsilon > 0$, there
exists $N \gg 1$ such that
\begin{equation*}
  {\widetilde m}_N = \sum_{j=-\infty}^{-N} \rho_j + \sum_{j=N}^\infty   \rho_j < \varepsilon .
\end{equation*}
Let $\varepsilon > 0$ to be fixed later. Then we propose a mass
preserving modification $\rhobt$ of $\rhob$ define by
\begin{equation*}
  \rhobt = (\vec{0}, \rho_{-N}, \rho_{-N+1}, \dots, \rho_{-1}, \rho_0+{\widetilde m}_N  , \rho_1, \dots, \rho_{N}, \vec{0} )\,.
\end{equation*}
Direct computation shows that 
\begin{align*}
H( \rhobt ) &= H( \rhob ) + {\widetilde m}_N ( \rho_0- 2 \rho_1 - 2 \rho_{-1}) + \frac12 {\widetilde m}_N^2 \\
& + 2 \rho_{N} \rho_{N+1} + 2 \rho_{-N} \rho_{-N-1} \\
&  + \sum_{j=N+1}^\infty \rho_j (2 \rho_{j+1} - \rho_j) + \sum_{j=- \infty}^{-N-1} \rho_j (2 \rho_{j-1} - \rho_j).
\end{align*}
By the above assumptions,
\begin{align*}
&\frac12 {\widetilde m}_N^2 + 2 \rho_{N} \rho_{N+1} + 2 \rho_{-N} \rho_{-N-1}  \\
& \hspace{14mm} + \sum_{j=N+1}^\infty \rho_j (2 \rho_{j+1} - \rho_j) + \sum_{j=-\infty}^{-N-1} \rho_j (2 \rho_{j-1} - \rho_j) < \frac{13}{2} \varepsilon^2,
\end{align*}
and hence
\begin{equation*}
  H( \rhobt ) <  H( \rhob ) + {\widetilde m}_N ( \rho_0- 2 \rho_1 - 2 \rho_{-1})  + \frac{13}{2} \varepsilon^2.
\end{equation*}
Since $\rho_0 \leq \frac56 (\rho_1 + \rho_{-1})$, we have that
\begin{equation*}
  ( \rho_0- 2 \rho_1 - 2 \rho_{-1}) \leq \frac56 (\rho_1 + \rho_{-1}) - 2 (\rho_1 + \rho_{-1}) < 0.
\end{equation*}
Hence, for $\varepsilon > 0$ sufficiently small, we will have reduced the energy by taking a compactly supported function. This completes the proof of the first claim and thus the proof of the lemma.
\end{proof}

We need to argue that in fact, the three mode solution will be a global energy minimizer.

\begin{lem}
\label{lem:5modes}
Any  minimize of the energy $H$ with finite mass is supported on $5$ or fewer  nodes.  
\end{lem}

\begin{proof}
  Using Lemma \ref{lem:compact}, we can restrict our attention to a
  compact solution $\rhob = \{ \rho_j \}_{j=-N}^N$ for some finite
  $N$ with at least $\rho_N$ or  $\rho_{-N}$ strictly positive.

  By  \Cref{lem:dec}, we can assume that the $\rhob$ is monotone non-increasing as we head away from the central mode. In particular, this last property implies that $\rho_j \geq 0$ on $|j|\leq N$, or
    \begin{align*}
    \rhob = (\rho_{-N}, \rho_{-N-1},\dots,\rho_{-2},\rho_{-1} , \rho_0 , \rho_1 , \rho_2, \dots, \rho_{N-1} , \rho_N).
  \end{align*}
  
  We define $\rhobt$ removing the outermost modes and distributing their mass over the middle three modes; namely,
  \begin{align*}
    \rhobt = (\rho_{-N-1},\dots,\rho_{-1} + \frac12 \rho_{-N}, \rho_0 + \frac12 \left( \rho_{-N} + \rho_N \right), \rho_1 + 
    \frac12 \rho_N, \dots, \rho_{N-1}).
  \end{align*}
  Notice that $\rhob$ and $\rhobt$ have the same mass. Calculating, we have
  \begin{multline*}
      H(\rhobt)  = H(\rhob)  - \frac12 \left[  \rho_{-1}^2 + \rho_0^2 + \rho_1^2 + \rho_{-N}^2 + \rho_{-N}^2 \right] \\
       + \frac12 \left[  \big( \rho_{-1} + \frac12 \rho_{-N} \big)^2 +   \big( \rho_{0} + \frac12 (\rho_N + \rho_{-N}) \big)^2  +  \big( \rho_{1} + \frac12 \rho_{N} \big)^2   \right] \\
       + 2 \left[ \rho_{-N} \rho_{-N+1} +  \rho_{N-1} \rho_N + \rho_{-2} \rho_{-1} + \rho_{-1} \rho_0 + \rho_{0} \rho_{1} + \rho_1 \rho_2 \right] \\
       - 2 \left[  \rho_{-2} \big(  \rho_{-1} + \frac12 \rho_{-N} \big)  +  \big(  \rho_{-1} + \frac12 \rho_{-N} \big)  \big(  \rho_{0} + \frac12 (\rho_{-N} + \rho_N) \big) \right. \\
       \left. +  \big(  \rho_{0} + \frac12 (\rho_{-N} + \rho_N)
        \big) \big(  \rho_{1} + \frac12 \rho_{N} \big) + \big(
        \rho_{1} + \frac12 \rho_{N} \big) \rho_2 \right].
  \end{multline*}
  Simplifying produces
  \begin{multline*}
    H(\rhobt)      = H(\rhob) - \frac38 \rho_{-N}^2 - \frac38 \rho_N^2 - \frac38 (\rho_N + \rho_{-N})^2 \\
      - \frac12 (\rho_{-N} + \rho_N) ( \rho_1 + \rho_{-1} - \rho_0 ) \\
      - \rho_N ( \rho_2 + \rho_0 + \frac12 \rho_{-1} - 2\rho_{N-1}) \\
      - \rho_{-N} ( \rho_{-2} + \rho_0 + \frac12 \rho_{1} - 2\rho_{-N+1}) .
  \end{multline*}
  We now claim that the sum of the five terms subtracted from $H(\rhob)$ is strictly positive. We will see that some of the terms are strictly positive while others are simply non-negative. Since at the start, we chose $\rho_{+N}+\rho_{-N}>0$ we know that $\rho_{-N}^2 + \rho_N^2+ (\rho_N + \rho_{-N})^2$ is  strictly positive. Using Lemma \ref{lem:5over3} we see that $\rho_1 + \rho_{_1} - \rho_0 >0$ so $(\rho_{-N} + \rho_N) ( \rho_1 + \rho_{-1} - \rho_0 )$ is also strictly positive.  Lemma \ref{lem:dec} implies that $\min(\rho_2, \rho_0) \geq \rho_{N-1}$ and  $\min(\rho_{-2}, \rho_0) \geq \rho_{-N+1}$ so the last two terms of interest are non-negative because
  \begin{align*}
    \rho_N( \rho_2 + \rho_0 + \frac12 \rho_{-1} - 2\rho_{N-1})\geq   \frac12  \rho_N\rho_{-1} &\geq 0, \\
    \rho_{-N}( \rho_{-2} + \rho_0 + \frac12 \rho_{1} - 2\rho_{-N+1}) \geq \frac12 \rho_{-N} \rho_{1}&\geq 0\,.
  \end{align*}
Combining these observations we see that $H(\rhobt) <   H(\rhob)$ meaning that  we can decrease the energy by removing mass from the outermost modes  provided $N-1 \geq 2$ or $N \geq 3$. 
\end{proof}

\begin{thm}\label{thm:3ModeRho}
 The global minimizer of $H$ with fixed mass is $3$-mode solution described in \eqref{enmin}.
\end{thm}

\begin{rem}\label{rem:b_from _rho}
  \Cref{thm:3ModeRho} gives the three-mode minimizer in terms
  of the $\rho$ variables with group phase velocity $\omega$.  This translates to
$\bb(t)=\bb^*e^{i \omega t+i\theta}$ with $\bb^*=(\cdots,0,b_{-1},b_0,b_1,0,\cdots)$,
$b_{-1}=b_{1}=\sqrt{\frac{3m}{11}}$, $b_0=\sqrt{\frac{5m}{11}}$, and
$\omega=\frac{7 m}{22}$. Here, $\theta$ captures the phase at time zero.
\end{rem}

\begin{rem}
  \Cref{rem:b_from _rho} give the minimizing time evolving solutions
  centered on the zero lattice site. If one only looks at a fixed time
  $t$ then the minimizing solution centered on the zero lattice
  position are given by $e^{i\theta} \bb^*$. More generally, the
  minimizers centered on the $k$th lattice site are given by
  $e^{i\theta}\bb_k^*$ where
  $\bb^*=(\cdots,0,b_{k-1},b_k,b_{k+1},0,\cdots)$. These are the
  minimizing solutions given in \Cref{mainthm}.
\end{rem}

\begin{proof}[Proof of \Cref{thm:3ModeRho}]
Now, we proceed to prove that the optimizer is the three mode solution
described above.  

Since the solution is restricted to at most five modes from Lemma \ref{lem:5modes}, we may argue similarly to Lemma \ref{lem:5over3} to in fact demonstrate that the optimizer has at most 4 modes,  To see this, assume that the optimizer is five modes, 
\begin{equation*}
  \rhob =( \rho_{-2},\rho_{-1},\rho_0, \rho_1, \rho_2).
\end{equation*}
Assume that $\rho_{-2},\rho_2 > 0$.  For $0 < \varepsilon < \min \{ \rho_{-2},\rho_2 \}$, define
\begin{equation*}
  \rhobt = (\rho_{-2}-\varepsilon,\rho_{-1},\rho_0+ 2\varepsilon, \rho_1, \rho_2 - \varepsilon).
\end{equation*}
We compute that
\begin{equation*}
    H(\rhobt) = H(\rhob) - (\rho_{-2} + \rho_2 + 2 (\rho_1 + \rho_{-1} - \rho_0) ) \varepsilon + 3 \varepsilon^2.
\end{equation*}
Hence, we can decrease the energy by moving mass out of at least one of the nodes $\rho_{-2}$ or $\rho_2$.  

Hence, we can restrict ourselves to optimizers over $4$ modes, indexed say $\rho_0, \rho_1,\rho_2$ with the dimensional reduction $\rho_3 = 1- \rho_0 - \rho_1 - \rho_2$ due to the mass constraint.  From here, we may compute the Jacobian exactly over the reduced set of variables and compute to get candidates for the minimizer as the $4$-mode solution, the three mode solution and the $1$-mode solution. On the boundary of the constraint that $\rho_j \geq 0$ for all $j$, the four mode boundary has at most three modes.  Indeed, 
\begin{align*}
  &   H(\rho_0,\rho_1,\rho_2,1-\rho_0-\rho_1 - \rho_2) = \\ 
  & \hspace{1cm} \tfrac12 + \rho_0^2 + \rho_1^2 + 3 \rho_2^2 - \rho_0 - \rho_1 - 3\rho_2 - \rho_0 \rho_1 + 3 \rho_2 \rho_0 + \rho_1 \rho_2.
\end{align*}
The Jacobian over $\rho_0,\rho_1,\rho_2$,
\begin{align*}
    \nabla H = (2 \rho_0 -1  - \rho_1 + 3 \rho_2, 2 \rho_1 -1 - \rho_0 + \rho_2, 6 \rho_2 - 3 + 3 \rho_0 + \rho_1).
\end{align*}
gives one critical point away from the constraint boundary, namely the $4$-mode solution
\[
\rhob = \left( \tfrac18, \tfrac38, \tfrac38,\tfrac18 \right).
\]

Once we go to the boundary of the constraint space, we have at least one mode $0$, so to consider the boundary values, we now restrict to the $3$-mode solution.

First, for any solution with at most three modes, we establish that the solution is symmetric around the maximum.   Let
  $\rhob = (\rho_{-1} , \rho_0, \rho_1)$.  In this notation, the fact that $\rhob$ has mass one translates into
  \begin{align*}
    \rho_{-1} + \rho_{1} + \rho_0 = 1,
  \end{align*}
and the energy can be written as
\begin{align*}
  H(\rhob) = \tfrac12 (\rho_{-1}^2 + \rho_0^2 + \rho_1^2) - 2 (\rho_{-1} + \rho_1) \rho_0.
\end{align*}
Defining a $\rhobt = \left( \frac{\rho_1+ \rho_{-1}}{2}, \rho_0, \frac{\rho_1+ \rho_{-1}}{2} \right)$, we have
\begin{align*}
H(\rhobt) & =  \tfrac14 (\rho_1+ \rho_{-1})^2 + \tfrac12 \rho_0^2 - 2 (\rho_1+ \rho_{-1})( \rho_0) \\
          & = H(\rhob) - \tfrac14 (\rho_1 - \rho_{-1})^2,
\end{align*}
hence the $3$-mode optimizer may be taken to be symmetric, i.e. $\rho_{-1} = \rho_1$.

We now have that $\rho_0 + 2 \rho_1 = 1$, hence
letting $0 < \rho_0 = \alpha \leq 1$, we have
\begin{equation*}
    H(\rho) = \tfrac12 \alpha^2 + \tfrac14 (1-\alpha)^2 - 2 \alpha 
  (1-\alpha) = 1 - \tfrac52 \alpha + \tfrac{11}{4} \alpha^2,
\end{equation*}
which has a minimum when $\alpha = \frac{5}{11}$.  Using \eqref{eqn:Hminval}, for a given mass $m$, we have that for
the three mode solution, denoted say $\rhob^*$, that
\begin{equation*}
  H(\rhob^*) = -\tfrac{7}{22} m^2 ,
\end{equation*}
while recall that for the one mode solution, the energy is $\tfrac{1}{2} m^2$ and the $4$-mode energy is $-\tfrac{5}{16} m^2$, which is larger.
\end{proof}

\section{Linearizing about the 3-mode Minimizers}
\label{Sec:JacHess}

To approximate the invariant measure in a neighborhood of the
optimizer, we need to understand how the Hamiltonian behaves in a
neighborhood.  To that end, we compute the Jacobian and Hessian
around the optimizer. To assist in the computation, we begin by
finding the explicit form of the Lagrange multiplier for any of the
three-mode minimizer. We will initially consider the minimizer $\bb^*$
which centered on the zero  lattice site with phase zero. 

\subsection{The Lagrange Multiplier}
Recall from \eqref{eq:discreteLap} in Section~\ref{Sec:Stationary},
that $\omega$ denoted the Lagrange Multiplier associated with
minimizing the energy for a given mass constraint.

Let $\bb^*$ be one of the three mode in phase solution that minimizing the energy for a mass
$m$. Let $\omega(m)$ be the Lagrange Multiplier associated with its
minimization problem.

By direct calculation, we observe that
\begin{equation}
\label{eqn:lagmult}
\nabla H(\bb^*) = - \tfrac{14  }{11}m \bb^*.
\end{equation}
Also, note that for any vector $\pm{\bm{\eta}} \in  \mathbb{C}^{2N+1}$
such that $| \bb^* + \bm{\eta}|^2 = m$, we have the identity
\begin{equation}
\label{eqn:constraint}
\langle \bb^*, \bm{\eta} + \overline{ \bm{\eta} } \rangle + | \bm{\eta} |^2 = 0.
\end{equation}
Hence, importantly 
\begin{multline}
\label{eqn:LM}
   {\Re}(\nabla H (\bb^*) \eta)  =  \nabla H (\bb^*) ( \bm{\eta}/2 + \overline{ \bm{\eta} }/2 )\\
    = - \tfrac{7}{11}m \bb^* \cdot ( \bm{\eta} + \overline{ \bm{\eta} } ) = \tfrac{7}{11}m |\eta|^2.
\end{multline}
These calculations allow us to exactly construct a Lagrange multiplier for the constrained optimization problem of the energy on the sphere.

\subsection{The Hessian Calculation}\label{sec:Hessian}

We will now prove the results that have been stated in
\Cref{thm:Hessprops}.
We will denote by $\calH(\bb)$ the Hessian of $H$ evaluated at the
point $\bb$. As before, let $\bb^*_0$ be the minimizing state given by  $b_{-1} =
\sqrt{ \frac{3m}{11}} = b_1$, $b_0 = \sqrt{ \frac{5m}{11}}$ on the
modes ${-2,-1,0,1,2}$.
In the complex coordinates $\alpha_j + i \beta_j = b_j$,
$\calH(\bb^*_0)$ consists of a non-zero $10 \times 10$ sub-matrix
surrounded by zeroes. Here we have ignored the fact that we are
actually restricted to the sub-manifold of constant mass $m$.   This
will be dealt with using a Lagrange multiplier later.

We now compute this non-zero, sub-matrix of the Hessian using the coordinate
$\{ \alpha_{-2}, \beta_{-2}, \alpha_{-1}, \beta_{-1}, \dots, \alpha_2,
\beta_2 \}$ which produces
\begin{align*}
\calH(\alpha_j, \alpha_j) &=  6 \alpha_j^2 + 2 \beta_j^2 - 4( \alpha_{j-1}^2 + \alpha_{j+1}^2) + 4( \beta_{j-1}^2 + \beta_{j+1}^2) , \\
 \calH(\beta_j, \beta_j) &=  6 \beta_j^2 + 2 \alpha_j^2 + 4( \alpha_{j-1}^2 + \alpha_{j+1}^2) - 4( \beta_{j-1}^2 + \beta_{j+1}^2) , \\
 \calH(\alpha_j, \beta_j) &=  4 \alpha_j \beta_j - 8 \alpha_{j+1}  \beta_{j+1} - 8 \alpha_{j-1}  \beta_{j-1} ,\\
 \calH(\alpha_j, \alpha_{j+1}) &=  - 8 \alpha_j \alpha_{j+1} - 8 \beta_j \beta_{j+1} , \\
 \calH(\alpha_j, \beta_{j+1}) &=  8 \beta_{j+1} \alpha_j - 8 \alpha_{j+1} \beta_j ,  \\
 \calH(\beta_j, \beta_{j+1}) &= - 8 \beta_j \beta_{j+1} - 8 \alpha_{j} \alpha_{j+1} .
\end{align*}
In terms of our explicit (real) stationary solution, we have the following $10 \times 10$ matrix for the Hessian $\calH( \bm{\alpha},  \bm{\beta})$ with the vectors $\bm{\alpha} = ( \alpha_{-2}, \alpha_{-1}, \dots, \alpha_1, \alpha_2)$, and $\bm{\beta} = (\beta_{-2}, \beta_{-1}, \dots, \beta_1, \beta_2)$ defined on the modes $\{ -2,-1,0,1,2 \}$

\begin{align}
\label{eqn:Hessform}
&\calH (\bb^*_0) = m \left[  \begin{array}{cc}
\H_{\bm{\alpha}} & 0 \\
0 & \H_{\bm{\beta}}
\end{array} \right],
\end{align}
where $\H_{\bm{\alpha}}$ is given by
{\tiny 
\begin{align}
\label{eqn:HessA}  
 \left[  \begin{array}{ccccc}
-4 \left( \frac{3}{11} \right) & 0 & 0 & 0 & 0  \\
* & 6 \left( \frac{3}{11} \right)  - 4  \left( \frac{5}{11} \right)    & -8  \left( \frac{\sqrt{15} }{11} \right)  & 0  & 0 \\
* & *& 6 \left( \frac{5}{11} \right)  - 4  \left( \frac{6}{11} \right)    &  -8  \left( \frac{\sqrt{15} }{11} \right)  & 0  \\
* & * & * & 6 \left( \frac{3}{11} \right)  - 4  \left( \frac{5}{11} \right)  & 0  \\
* & * & * & * & -4 \left( \frac{3}{11} \right) 
\end{array} \right]
\end{align}
}
and $H_{\bm{\beta}}$ given by
{\tiny  
\begin{align}
 \label{eqn:HessB}  
\left[  \begin{array}{ccccc}
4 \left( \frac{3}{11} \right)  & 0 & 0 & 0 & 0 \\
 * & 2 \left( \frac{3}{11} \right) + 4  \left( \frac{5}{11} \right)   & -8 \left( \frac{\sqrt{15}}{11} \right) & 0 & 0 \\
 * & * & 2 \left( \frac{5}{11} \right) + 4  \left( \frac{6}{11} \right)    &  -8 \left( \frac{\sqrt{15}}{11} \right)  & 0 \\
 * & * & * & 2 \left( \frac{3}{11} \right) + 4  \left( \frac{5}{11} \right)    & 0 \\
 * & * & * & * & 4 \left( \frac{3}{11} \right) \\
\end{array} \right].
\end{align}
} 

Above, the $*$ elements are completed through using that the matrix is self-adjoint.  It is however easy to see that this matrix is not positive definite by considering the first and last rows of the matrix $\H_{\bm{\alpha}}$.  However, in order to be at a minimizer, the mass constraint must be taken into account, which will be done later in the characterization of the measure.

In preparation for considering the constraint, it is useful to compute
that 
\begin{equation}
\label{eqn:H2bprod}
\nabla^2 H(\bb^*_0) \bb^* = -\tfrac{42}{11}m \bb_0^*
\end{equation}
and for all $0<\delta<1$, we have
\begin{equation}
H((1-\delta) \bb^*_0)=(1-\delta)^4  H(\bb^*_0) > H(\bb^*_0) = -\tfrac{7}{22}m^2.
\end{equation}
 We can also compute the eigenvalues of $\tfrac{14}{11} I + \H_{\bm{\alpha}}$, which are as stated in Theorem \ref{thm:Hessprops} given by
\begin{equation*}
 -\tfrac{28}{11},\tfrac{2}{11},\tfrac{2}{11},\tfrac{12}{11},\tfrac{60}{11}
\end{equation*}
 with the eigenvectors being
 \begin{equation*}
   \bb_0^*,\be_{-2},\be_2,\be_1-\be_{-1},\sqrt{\tfrac53} (\be_1 + \be_{-1}) - 2 \be_0
 \end{equation*}
 respectively using the notation of $\be_j$ for standard basis vectors such that $\be_j$ has $e_j = 1$ and $e_k = 0$ for all $k \neq 0$.  
 
 Similarly, we have as stated in Theorem \ref{thm:Hessprops} given by the eigenvalues of $\tfrac{14}{11} I + \H_{\bm{\beta}}$ are 
 \begin{equation*}
    0,\tfrac{26}{11},\tfrac{26}{11},\tfrac{40}{11},8 
 \end{equation*}
 with respective eigenvectors
 \begin{equation*}
    \bb_0^*,\be_{-2},\be_2,\be_1-\be_{-1},\sqrt{\tfrac53} (\be_1 + \be_{-1}) - 2 \be_0.
 \end{equation*}
Moving from the real parametrization to complex, all the results in \Cref{thm:Hessprops} are now established.  We note that the perturbation in the direction of $i \bb_0 = (0,\bb_0)$ will
 generate rotation in phase and generate a $0$ direction for the
 Hessian.

\section{A Gaussian approximation}
\label{Sec:Gaussian}

In the introduction, we described a Gaussian approximation using a function $G(\bb)$ in \eqref{eq:G} that was quadratic in the difference $\Proj_{\bb^*}^\perp(\bb)=\bb-\Proj_{\bb^*}(\bb)$. In the forthcoming analysis, it is convenient to use a
version that measures the distance to each $B_k^*$ separately. To this
end, let
\begin{align}\label{eq:bk}
  \widehat \bb_k^*(\bb) = \argmin_{\bb^* \in B_k^*} | \bb - \bb^*|
\end{align}
and define  
$$\Proj_{\bb_k^*}^\perp(\bb)  \eqdef \bb- \Proj_{\bb_k^*}(\bb)$$ 
where
\begin{equation}
  \Proj_{\bb_k^*}(\bb)=\big \langle \bb,  \tfrac{\widehat \bb_k^*(\bb)}{|\widehat \bb_k^*(\bb)|}  \big \rangle \tfrac{\widehat \bb_k^*(\bb)}{|\widehat \bb_k^*(\bb)|}
\end{equation}
is the orthogonal projection of $\bb$ onto $\widehat \bb_k^*$.
 Then, we define the function $G_k\colon \C^{2N+1} \rightarrow \R$  by
\begin{equation} \label{eq:Gk}
   G_k(\bb) = \tfrac12  \big\langle \big(\tfrac{14}{11}M(\bb) I
  +\nabla^2 H(\widehat \bb_k^*(\bb))\big) \Proj_{\bb_k^*}^\perp(\bb) ,  \Proj_{\bb_k^*}^\perp(\bb)  \big\rangle
\end{equation}
with the same caveats as in and around
\Cref{rem:non-unique_distance}. If we define the $\epsilon$-neighborhood $B_k^{*\epsilon}(m)=\{ \bb \in
\C^{2N+1} : |\bb-B_k^*(m)| < \epsilon\}$, then $G(\bb)=G_k(\bb)$ and $ \widehat\bb^*(\bb) =  \widehat\bb_k^*(\bb)$  for all
$\bb \in B_k^{*\epsilon}$ if 
$\epsilon$ sufficiently small. Furthermore the non-uniqueness issues in
the definition of  $ \widehat\bb^*$,  $ \widehat\bb_k^*(\bb)$, $G$ and $G_k$ do not
occur when restricting to $B_k^{*\epsilon}$ for $\epsilon$ sufficiently small.

Observe that these quadratic forms have the same rotational invariance 
$G_k( e^{i \theta} \bb)=G_k(\bb)$ as the energy $H$. We now define the 
family of unnormalized, Gaussian-like measures $\widehat\gamma^N_{\beta,m}$ by 
\begin{align}\label{eq:gamma-hat}
  \widehat\gamma_{\beta,m}^N(\dbb) \eqdef  \sum_{k={-N+1}}^{N-1}  e^{- \beta G_{k}(\bb)}  \Gamma_m(\dbb)\,. 
\end{align}

The following two Corollaries on the Gaussian-like measure we just
defined are relatively standard and not surprising. We sketch the
proofs of each at the end of \Cref{sec:LocalCoordinates}.
\begin{corollary}\label{rem:contration}
  For large $\beta$, the measure $\gamma_{\beta,m}^N$ is 
  concentrated around the set of all three mode minimizers $B^*(m)$.
  More precisely  for any $\epsilon>0$, if $B^{*\epsilon}(m) = \bigcup_k
  B_k^{*\epsilon}(m)$ then
  \begin{align*}
  \frac{\widehat \gamma_{\beta,m}^N(B^{*\epsilon})}{\gamma_{\beta,m}^N(S(m)) }
    \xrightarrow{\beta \rightarrow \infty}  1 \quad\text{and}\quad 
    \frac{1}{\beta^N} \widehat\gamma_{\beta,m}^N(  S(m)) \xrightarrow{\beta 
    \rightarrow \infty} 2\pi (2N-1) \,.
  \end{align*}
\end{corollary}

\begin{corollary}\label{c:GareTheSame}
For any $A \subset S(m)$, 
  \begin{align}\label{eq:diff_gamma}
\frac{ \widehat \gamma_{\beta,m}^N(A)}{\widetilde \gamma_{\beta,m}^N(A)}
    \xrightarrow{\beta \rightarrow \infty}  1 \quad\text{and}\quad
    \frac{ \widehat \gamma_{\beta,m}^N(A)}{ \gamma_{\beta,m}^N(A)}
    \xrightarrow{\beta \rightarrow \infty}  1
  \end{align}
  where $\gamma_{\beta,m}^N(\dbb)= e^{-\beta G(\bb)}\Gamma_m(\dbb)$,
  see \eqref{eq:gaussian} with $G$ defined in
    \eqref{eq:G}, and  $\widetilde \gamma_\beta^N(\dbb)=
    e^{-\beta \widetilde G(\bb)} \Gamma_m(\dbb) $ with $\widetilde G(\bb) = \sum_{k=-N+1}^{N-1}
    G_{k}(\bb)$. Furthermore, the error terms in \eqref{eq:diff_gamma}
    decay to zero like $e^{- c \beta}$ for some $c>0$ as $\beta
    \rightarrow \infty$.
  \end{corollary}

  \subsection{Properties the  Gaussian Measures }
We now make some observations on the structure of the ``Gaussian''
measures defined above. In light of \Cref{c:GareTheSame}, we can
consider any of the measures $\gamma_{\beta,m}^N$, $\widehat
\gamma_\beta^N$, or $\widetilde \gamma_{\beta_m}^N$ if we are interested
in the behavior for large $\beta$. In the following it is most
convenient to consider $\widetilde \gamma_{\beta,m}^N/Z$ where $Z=
\widetilde \gamma_{\beta,m}^N(S(m))$ is the normalization constant
needed to make the measure a probability measure. A random draw from
$\widetilde \gamma_{\beta,m}^N/Z$ can be constructed by first drawing
a $k$ from $\{-N+1,\dots,N-1\}$ with uniform probability.  One then
chooses configuration according to the probability measure with density
$e^{\beta G_k(\bb)}/Z_k$ where $Z_k=Z/(2N+1)$ is the correct
normalization constant to make the measure a probability measure.

To further explore the measure $e^{-G_k(\bb)}/Z_k$, we recall from
\eqref{eqn:Hessform}-\eqref{eqn:HessB} that all the entries of $\nabla^2 H(e^{i \theta}\bb^*_k) =
0$ outside the $5$ nearest neighbor modes centered on $k$, namely
$\mathcal{N}_k=\{-2,k-1,k,k+1,k+2\}$. This implies that only the nodes
in $\mathcal{N}_k$ are correlated while the notes outside of
$\mathcal{N}_k$ are independent, identically distributed complex Gaussian
random variables with mean zero and variance $\frac{11}{14}$. This is encoded by the fact that if
$\bm{\eta}$ is such that $\bm{\eta_j}=0$ for $j \in \mathcal{N}_k$
then 
\begin{equation*}
  \big\langle\big( \tfrac{14}{11} I + \nabla^2 H(\bb^*_k)\big)
\bm{\eta}, \bm{\eta}\big\rangle = \tfrac{14}{11}\sum_{j \not \in  \mathcal{N}_k}  |\bm{\eta_j}|^2.
\end{equation*}
We also see that the  global phase of the random state is uniform on $[0,2\pi]$.

\subsection{Defining the measure with intrinsic coordinates on the sphere}
\label{sec:intrinsicG}
Our definitions of the various Gaussian measure preceding section were
build on the intrinsic distance in $\C^{2N+1} \approx \R^{4N+2}$ and
not on the intrinsic distance in $S(m)$ to which our measures are
restricted. We now show that we could have equally defined our
Gaussian measures using the intrinsic distances in $S(m)$. We will see
that the resulting measures are equivalent as $\beta \rightarrow
\infty$, which is the regime where the Gaussian measure approximate $\mu^N_{\beta,m}$.

We start by recalling that for any ${\bf x}, {\bf y } \in S(m)$, the
log-map 
\[
{\rm Log}_{\bf x}\colon S(m) \rightarrow T_{\bf
  x}S(m).
  \]
  While $T_{\bf  x}S(m)$ is just an abstract copy of
$\R^{4N+1}$, we can also view it as $4N+1$-dimensional plane contained
in $S(m)$ viewed as a copy of $\R^{4N+1}$. This $4N+1$-dimensional
plane contains the origin and is perpendicular to ${\bf x}$. These two
facts define it uniquely. 

In these extrinsic coordinates on $S(m)$, the Log Map ${\bf y}\mapsto
{\rm Log}_{\bf x}{\bf y}$ can be written as
\begin{equation}
    \label{eq:logmap}
    {\rm Log}_{\bf x} ({\bf y}) = d({\bf x},{\bf y}) \frac{{\bf y} -
      \frac{\langle {\bf x}, {\bf y} \rangle}{m} {\bf x}}{|{\bf y} -  \frac{\langle {\bf x}, {\bf y} \rangle}{m}  {\bf x}|}, \ \ d({\bf x},{\bf y}) := \sqrt{m} \arccos \left( \frac{\langle {\bf x}, {\bf y} \rangle}{m} \right).
\end{equation}
Here, the distance between points ${\bf x},{\bf y} \in S(m)$ has intrinsic distance $ d({\bf x},{\bf y}) =\sqrt{m}  \arccos
\left( \frac{ \langle {\bf x},{\bf y} \rangle }{m} \right)$ where $\langle {\bf x},{\bf y}
\rangle/m$ is the cosine of the angle between the two
vectors in $C^{2N+1}$.

Without loss of generality, we may assume that $\bb \in S(1)$. We now define
\begin{equation*}
  \GG_k(\bb) = \tfrac12  \big\langle \big(\tfrac{14}{11} I
  +\nabla^2 H(\widehat \bb_k^*(\bb))\big)  {\rm Log}_{\bb_k^*(\bb)}
  \bb ,    {\rm Log}_{\bb_k^*(\bb)}\bb
  \big\rangle .
\end{equation*}
In analogy with the previous measures we also $\widetilde \GG(\bb)=
\sum_{k=-N+1}^{N-1}\GG_k(\bb)$ and 
\begin{equation*}
  \GG(\bb) = \tfrac12  \big\langle \big(\tfrac{14}{11} I
  +\nabla^2 H(\widehat \bb^*(\bb))\big)  {\rm Log}_{\widehat \bb^*(\bb)}
  \bb ,    {\rm Log}_{\widehat \bb^*(\bb)}\bb
  \big\rangle .
\end{equation*}
These new functions biased on the ${\rm Log}$ map all converge in a
strong way to the original functions as we restrict to a smaller
and smaller neighborhood of ${\bb^*(\bb)}$. In preparation, observe
that 
\begin{align*}
\Proj_{\bb_k^*}^\perp(\bb) &= \|\bb - \langle \bb, \widehat
                             \bb_k^*(\bb)\rangle \widehat
                             \bb_k^*(\bb)\| \frac{\bb - \langle \bb, \widehat
                             \bb_k^*(\bb)\rangle \widehat
                             \bb_k^*(\bb)}{ \|\bb - \langle \bb, \widehat
                             \bb_k^*(\bb)\rangle \widehat
                             \bb_k^*(\bb)\| },\\
 {\rm Log}_{\widehat \bb_k^*(\bb)}
  \bb &= d(   \bb_k^*(\bb), \bb) \frac{\bb - \langle \bb, \widehat
                             \bb_k^*(\bb)\rangle \widehat
                             \bb_k^*(\bb)}{ \|\bb - \langle \bb, \widehat
                             \bb_k^*(\bb)\rangle \widehat
                             \bb_k^*(\bb)\|}.
\end{align*}

We
will only need to consider  $\bb$ so that if  $\bb^*=\widehat
\bb_k^*(\bb)$ then   $d(\bb^*,\bb)< \delta$ for some small
$\delta>0$. Then there
exists a $t \in [0,1]$ so that we can express $\bb$ in terms of its projection onto $\bb^*$ and the orthogonal component:
$$\bb=  (1-t)\bb^*_k - \sqrt{2t - t^2}
{\bf \eta}$$ 
with $\langle {\bf \eta}, \bb^*_k  \rangle= 0$, $\|{\bf
  \eta}\|=1$ and $\cos( d(\bb^*_k,{\bf y}))  =  1-t$. Then we have
that  $d(\bb^*_k,{\bf y}) \approx \sqrt{2t} $.  Hence, these coordinate systems are approximately the same in a small neighborhood of $\bb^*$.  As in Corollaries \ref{rem:contration}, \ref{c:GareTheSame}, we can then state the following.

\begin{lem}\label{lem:intrisicGamma}
  If, in analogy to before, we define the measures 
  \[
  \ugamma_{\beta,m}^N(\dbb) \eqdef e^{-\beta
    \GG(\bb)}\Gamma_m(\dbb), \] 
    \[ \widehat \ugamma_{\beta,m}^N(\dbb)\eqdef
  e^{-\beta \widetilde \GG(\bb)} \Gamma_m(\dbb), \] 
  and
  \[ \widehat\ugamma_{\beta,m}^N(\dbb) \eqdef  \sum_{k={-N+1}}^{N-1}
  e^{- \beta \GG_{k}(\bb)}  \Gamma_m(\dbb),\] 
  then for any $A \subset S(m)$
 \begin{align}\label{eq:diff_ugamma}
\lim_{\beta \rightarrow \infty}\frac{  \ugamma_{\beta,m}^N(A)}{
   \gamma_{\beta,m}^N(A)}=
   \lim_{\beta \rightarrow \infty}\frac{ \widehat \ugamma_{\beta,m}^N(A)}{\widehat \gamma_{\beta,m}^N(A)}=\lim_{\beta \rightarrow \infty}   
    \frac{ \widetilde \ugamma_{\beta,m}^N(A)}{ \widetilde \gamma_{\beta,m}^N(A)}
 = 1 .
  \end{align}
\end{lem}

\section{Invariance under Phase Rotation}\label{sec:rotinv}

We now prove a general result that implies \Cref{lem:avg} from the
introduction when combined with fact that $H(e^{i\theta} \bb) =
H(\bb)$ for any $\theta \in [0,2\pi]$ and $\bb \in \C^{2N+1}$.
\begin{proposition}\label{prop:avgGeneral}
  Let $F\colon \C^{2N+1} \rightarrow \R$ satisfy the phase
  invariance property $F(e^{i\theta}\bb)= F(\bb)$ for any $\theta \in
  [0,2\pi]$. Then for any $\phi\colon \C^{2N+1}
  \rightarrow \R$ so that $\int \phi(\bb) e^{-F(\bb)} \Gamma_m(\dbb)
< \infty$, we have
\begin{align*}
  \int \phi(\bb) e^{-F(\bb)} \Gamma_m(\dbb) = \int \overline{\phi}(\bb) e^{-F(\bb)} \Gamma_m(\dbb),
\end{align*}
where
\begin{align*}
   \overline{\phi} (\bb) = \frac1{2\pi} \int_0 ^{2\pi} \phi(e^{i\theta}\bb)  \Gamma_m(\dbb).
\end{align*}
\end{proposition}
\begin{proof} Defining the rotation $R_\theta(\bb)=(e^{i\theta}\bb)$,
  the fact that  $F(R_\theta\bb) =  F(\bb)$ follows directly from 
  \eqref{e:hamiltonian}. Similarly,
  $\Gamma_mR^{-1}_\theta(d\bb)=\Gamma_m(d\bb)$ follows from that fact
  that $R_\theta(S(m))=S(m)$. Here the push-forward $\Gamma_mR_\theta$
  of the measure $\Gamma_m$ is defined by
  $\Gamma_mR_\theta(A)=\Gamma_m(R_\theta(A))$ for any set $A
  \subset S(M)$. These two invariants imply that
  \begin{align*}
    \int_{S(m)} \phi(\bb) e^{-
   \beta  F(\bb)  } \Gamma_m(\dbb)&= \int_{S(m)} \phi(\bb) e^{-
                                     F(R_\theta(\bb))  } \Gamma_mR_\theta(\dbb)\\
   & = \int_{S(m)} \phi(R_{-\theta}(\bb)) e^{-
                                    F(\bb)  } \Gamma_m(\dbb).
  \end{align*}
  Integrating the initial expression on the left hand side and the
  final expression on the right hand side in $\theta$ from 0 to $2\pi$
  and then dividing by $2\pi$ leaves the   left hand side unchanged
  and produces the claimed average expression on the right hand side.
\end{proof}

\begin{lem}\label{lem:RotDerivative}
  The for all $\theta \in [0,2 \pi]$, $\bb, \bxi \in \C^{2N+1}$, we have 
  \begin{align*}
     & H(e^{i \theta}
  \bb)=H(\bb), \ \ \nabla H(e^{i\theta}\bb)[e^{i\theta}
  \xi]=\nabla H(\bb)[\bxi], \\
  & \nabla^2
  H(e^{i\theta}\bb)[e^{i\theta} \bxi, e^{i\theta}
  \bxi]=\nabla^2 H(\bb)[\bxi, \bxi], \\
  &   \nabla^3
  H(e^{i\theta}\bb)[e^{i\theta} \bxi, e^{i\theta}
  \bxi,  e^{i\theta} \bxi]=\nabla^3 H(\bb)[\bxi, \bxi,\bxi], \\
  & \nabla^4
  H(e^{i\theta}\bb)[e^{i\theta} \bxi, e^{i\theta}
  \bxi,  e^{i\theta} \bxi, e^{i\theta} \bxi]=\nabla^4 H(\bb)[\bxi,\bxi, \bxi,\bxi].
  \end{align*}
\end{lem}
\begin{proof}
  Since $H$ is a 4th order polynomial 
  for any $\lambda \in \R$, $\bb, \xi \in \C^{2N+1}$, we have the following equality
  \begin{multline}\label{lem:Hexpansion}
    H[\bb + \lambda \bxi ]=  H(\bb ) +\lambda \nabla H(\bb)[\bxi]
    +\lambda^2 \nabla H(\bb)[\bxi, \bxi] \\+ \lambda^3 \nabla^3 H(\bb)[\bxi, \bxi,\bxi] + \lambda^4 \nabla^4 H(\bb)[\bxi, \bxi,\bxi,\bxi] \,.
  \end{multline}
We begin by observing that  $H[e^{i\theta}\bb + \lambda e^{i\theta}\bxi ]=  H[\bb + \lambda
\bxi ]$. Applying \eqref{lem:Hexpansion}  to both sides of this
equality and equating powers of $\lambda$ produces the quoted  result.
\end{proof}

\begin{lem}\label{l:prop_inv}
  For all $\theta \in [0,2\pi]$, $\bb \in \C^{2N+1}$, $k \in [-N+1,N-1]$, we have
  $\widehat \bb^*(\bb) =e^{-i \theta}  \widehat \bb^*(e^{i \theta}
  \bb)$, $\widehat \bb^*_k(\bb) = e^{-i \theta} \widehat \bb^*_k(e^{i
    \theta} \bb)$, as well as
    \begin{align*}
      &   \Proj_{\bb^*}(\bb) 
  =e^{-i\theta}\Proj_{\bb^*}(e^{i\theta}\bb), \ \Proj_{\bb^*_k}(\bb)
  =e^{-i\theta}\Proj_{\bb^*_k}(e^{i\theta}\bb), \\
  & \Proj_{\bb^*}^\perp(\bb) 
  =e^{-i\theta}\Proj_{\bb^*}^\perp(e^{i\theta}\bb), \Proj_{\bb^*_k}^\perp(\bb) =e^{-i\theta}\Proj_{\bb^*_k}^\perp(e^{i\theta}\bb).
    \end{align*}
\end{lem}
\begin{proof}
  First observe that $B_k^*=e^{i\theta}B_k^*\eqdef\{ e^{i\theta} \bb:
  \bb \in B_k^* \}$ for all $\theta \in [0,2\pi]$ and
  $k=-N+1,\dots,N-1$. Hence
  \begin{align*}
    \widehat \bb^*(e^{i\theta}\bb) &= \argmin_{\bb \in B^*} | e^{i
    \theta} \bb - \bb| \\ &= \argmin_{\bb \in B^*} | e^{i
    \theta} \bb - e^{i \theta} \bb|=  e^{i\theta} \widehat \bb^*(\bb) .
  \end{align*}
  A completely analogous calculation shows that $\widehat
  \bb^*_k(e^{i\theta}\bb)= e^{i\theta} \widehat \bb^*_k(\bb)$ for all
  $k=-N+1,\dots,N-1$. Next, we see that
  \begin{align*}
    \Proj_{\bb^*}(e^{i\theta}\bb) &= \big \langle e^{i\theta}\bb,  \tfrac{\widehat
                                    \bb^*(e^{i\theta}\bb)}{|\widehat
                                    \bb^*(e^{i\theta}\bb)|} \big
                                    \rangle \tfrac{\widehat
                                    \bb^*(e^{i\theta}\bb)}{|\widehat
                                    \bb^*(e^{i\theta}\bb)|}\\
    &=\big \langle e^{i\theta}\bb,   e^{i\theta} \tfrac{\widehat
                                  \bb^*(\bb)}{|\widehat
                                    \bb^*(\bb)|} \big
                                    \rangle e^{i\theta}\tfrac{\widehat
                                    \bb^*(\bb)}{|\widehat
      \bb^*(\bb)|}= e^{i\theta} \Proj_{\bb^*}(\bb) .
  \end{align*}
  This in turns implies that 
  \[
  \Proj_{\bb^*}^\perp(e^{i\theta}\bb)
  =e^{i\theta}\bb- \Proj_{\bb^*}(e^{i\theta}\bb)=
  e^{i\theta}\Proj_{\bb^*}^\perp(\bb)
  \]
  and
\[
\Proj_{\bb^*_k}^\perp(e^{i\theta}\bb)=e^{i\theta}\Proj_{\bb^*_k}^\perp(\bb)
\]
  again for $k=-N+1,\dots,N-1$.
\end{proof}

\begin{corollary}\label{c:Ginvariance}
  For every $\bb \in \C^{2N+1}$, $\theta \in [0,2 \pi]$, and
  $k=-N+1,\dots,N-1$, we have  $G(e^{i\theta} \bb)= G(\bb)$ and
  $G_k(e^{i\theta} \bb)= G_k(\bb)$. Additionally for any compactly
  supported, bounded  $\phi\colon \C^{2N+1}
  \rightarrow \R$, we have
\begin{align*}
 \gamma_{\beta,m}^N\phi  = \gamma_{\beta,m}^N 
  \overline{\phi}, \quad \widetilde \gamma_{\beta,m}^N\phi  =
  \widetilde \gamma_{\beta,m}^N 
  \overline{\phi},
  \quad\text{and}\quad
  \widehat \gamma_{\beta,m}^N\phi  = \widehat \gamma_{\beta,m}^N 
                                              \overline{\phi}\\
\end{align*}
where
\begin{align*}
   \overline{\phi} (\bb) = \frac1{2\pi} \int_0 ^{2\pi} \phi(e^{i\theta}\bb)  \Gamma_m(\dbb).
\end{align*}
\end{corollary}
\begin{proof}
  Using the results from \Cref{l:prop_inv} and the definition of $G$
  from \cref{eq:G}, we have 
\begin{align*}
   & G(e^{i\theta}\bb)= \\
   & \hspace{.5cm} \tfrac12 \big\langle \big(\tfrac{14}{11}M(e^{i\theta}\bb) I
  +\nabla^2 H(\widehat \bb^*(e^{i\theta}\bb))\big)
                      \Proj_{\bb^*}^\perp(e^{i\theta}\bb),\Proj_{\bb^*}^\perp(e^{i\theta}\bb) \big\rangle \\
  & \hspace{.5cm}= \tfrac12 \big\langle \big(\tfrac{14}{11}M(\bb) I
  +\nabla^2 H(\widehat \bb^*(e^{i\theta}\bb))\big)
    e^{i\theta}\Proj_{\bb^*}^\perp(\bb), e^{i\theta}
    \Proj_{\bb^*}^\perp(\bb) \big\rangle \\
    & \hspace{.5cm}= \tfrac12 \big\langle \big(\tfrac{14}{11}M(\bb) I
  +\nabla^2 H(\widehat \bb^*(\bb))\big)
    \Proj_{\bb^*}^\perp(\bb), 
    \Proj_{\bb^*}^\perp(\bb) \big\rangle= G(\bb).
\end{align*}
A similar calculation shows that $G_k
(e^{i\theta}\bb)=G_k(\bb)$. The last two results about integration
then follow from \Cref{prop:avgGeneral}.
\end{proof}

\section{Fluctuations at Low Temperature}
\label{Sec:Meas}

Finally, we now return to the measure $\mu_{\beta,m}^N$ defined in
\eqref{mu-beta-m}. Though we work on the finite domain $\C^{2N+1}$,
the results are the same for arbitrarily large $N$.
We begin by defining a convenient set of local coordinates that will
facilitate calculations. Then, using those coordinates and the
Hessian calculations from Section~\ref{sec:Hessian}, we define the
quadratic form that will be used for our Gaussian approximation.

\subsection{Approximating the measure} 

For a test function $\phi\colon S(m) \rightarrow \R$, we wish to consider
\begin{align}
\label{Gibbs}
&  e^{-\beta h_*(m)}  \int_{S(m)} e^{-\beta [H(\bb)-h_*(m)]}\phi(\bb) \Gamma_m( \dbb).
\end{align}
For $j_0 = -N+1, \dots, 0, \dots, N-1$ and $\theta_{k_0} = k_0 \pi/2$
for $k_0 = 0,1,2,3$, we decompose into a partition of unity centered
around the set of optimizers, which consists of phase rotated three mode solutions centered upon one of the lattice sites.  We will analyze the measure in a neighborhood around one of these designated fixed points, which we will without loss of generality take simply the real optimizer centered at $0$, i.e. $\bb_0^*$.  Recall that $h_*(m) = -\tfrac{7}{22} m^2$.
\begin{thm}
\label{thm:Gibbs_precise}
    For a bounded, smooth test function $\phi\colon S(m) \rightarrow \R$
    with compact support, we have 
        \begin{multline}\label{eq:conclusionCap}
      \frac{e^{\beta h_*(m)}}{\beta^{N}}     \mu_{\beta,m}^N\phi= \frac1{\beta^{N}}   
         \int_{S(m)} e^{-\beta (H(\bb)-h_*(m))} \phi(\bb) \Gamma_m( \dbb)  
         \\
    = \frac1{\beta^{N}}\widehat \gamma_{\beta,m}^N\phi + {\rm Error} (m,N,\beta,\phi) 
    \end{multline}
    for $\widehat \gamma_{\beta,m}^N$ defined in  \eqref{eq:gamma-hat}. Additionally, there exist  a constant $C(m,N) >0$,
    independent of $\phi$ and $\beta$, such that 
\begin{equation*}
   | {\rm Error} (m,N,\beta,\phi) | \leq  C(m,N)  \beta^{-1} \| \phi \|_{L^\infty (\Gamma_m)}. 
\end{equation*}
\end{thm}

\begin{rem}
We will demonstrate that the integral on the right and left hand side in \eqref{eq:conclusionCap} converge to non-zero, finite values for generic test function $\phi$
as $\beta \to \infty$, while the error term vanishes for all $\phi$ as $\beta \to \infty$.  We note that the constants involved depend upon $N$ and can indeed blow-up as $N \to \infty$. 
\end{rem}

\subsection{The Proof of Theorem \ref{thm:Gibbs_precise}}

As already discussed in \Cref{rem:contration}, as 
$\beta \rightarrow \infty$ we expect $\mu_{\beta,m}^N$ to concentrate 
inside of $S^\epsilon(m,0)$. To demonstrate this we need to decompose 
$\mu_{\beta,m}^N$ in local coordinates for $S(m)$ that are amenable to 
analysis.

For any $\delta \in (0,1)$, motivated by the Intrinsic Coordinates on the Sphere in Section \ref{sec:intrinsicG}, we define the $ \delta$-spherical cap about the circle of minimizers
$B_k^*=\{e^{i\theta}\bb_k: \theta \in [0,2\pi)\}$ by
\begin{align*}
  \SCap_k(\delta)  \eqdef \big\{ \bb \in S(m) : \sup_{\theta \in
  [0,2\pi]}\langle e^{i\theta}\bb_k^* , \bb\rangle > (1-\delta)\sqrt{m}\big\} .
\end{align*}
For $\delta$ sufficiently small, observe  that for any point $\bb \in
\SCap_k(\delta)$  there is a unique point $\hat \bb_k^*(\bb) \in B_k^*$ that
minimizes the distance between $\bb$ and $B_k^*$.
we will set $\hat \theta_k^*(\bb)=\arg \hat \bb_k^*(\bb)$. Henceforth,
we always assume that $\delta$ is chosen sufficiently small for this to
hold.

We then define the slice of $\SCap_k(\delta)$ at angle $\theta$ by and
\begin{align*}
  \SCap_k(\delta,\theta)   \eqdef \big\{ \bb \in  \SCap_k(\delta)  :
  \hat \theta_k^*(\bb)=\theta \big\}.
\end{align*}
We will denote by $\Upsilon_{m,\theta}(\dbb)$ the 
disintegration onto the partition 
\begin{align*}
 \{ \SCap_k(\delta,\theta)  \,|\, \theta \in [0,2\pi)\}
\end{align*}
of the restriction  of
$\Gamma_{m}(\dbb)$ to the open set $\SCap_k(\delta)$. The
normalization is chosen so that 
$\Gamma_{m}(\dbb)= \Upsilon_{m,\theta}(\dbb) d\theta$ on $\SCap_k(\delta)$.

\begin{lem}\label{lem:Gibbs_precise_cap}
 There exist a $\delta>0$ sufficiently small so that for any bounded, smooth 
  test function $\phi\colon S(m) \rightarrow \R$ with compact support 
  contained in $\SCap_k(\delta)$, we have 
        \begin{multline}\label{eq:conclusionCap1}
         \frac1{\beta^{N}}   
         \int_{\SCap_k(\delta)} e^{-\beta (H(\bb)-h_*(m))} \phi(\bb) \Gamma_m( \dbb)  
         \\
    =  \frac1{\beta^{N}}    \int_{\SCap_k(\delta)} e^{- \beta G_k(\bb)} \phi(\bb) \Gamma_m( \dbb)  + {\rm Error}  (m,\beta,N,\phi,k)
    \end{multline}
    for $G_k$ defined in \eqref{eq:G}. Additionally,  there exist 
    constants $C(m,N) >0$ independent of $\phi$, $k$ and $\beta$ such that 
\begin{equation*}
   | {\rm Error} (m,\beta,N,\phi,k) | \leq  C(m,N)  \beta^{-1} \| \phi \|_{L^\infty (\Gamma_m)}. 
\end{equation*}
\end{lem}

\Cref{lem:Gibbs_precise_cap} follows from the following even more
localized lemma.

\begin{lem}
\label{lem:Gibbs_precise_cap_slice}
    For any smooth, bounded function $\phi\colon S(m) \rightarrow \R$
    supported in $\SCap_k(\delta)$ for all $\theta \in
    [0,2\pi)$ and $\delta >0$ sufficiently small , we have 
        \begin{multline}\label{eq:conclusionCapSlice}
          \frac1{\beta^{N}}   
         \int_{ \SCap_k(\delta,\theta) } e^{-\beta (H(\bb)-h_*(m))} \phi(\bb) \Upsilon_{m,\theta}(\dbb)
         \\
    =    \frac1{\beta^{N}}    \int_{ \SCap_k(\delta,\theta) } e^{- \beta G_k(\bb)} \phi(\bb)  \Upsilon_{m,\theta}(\dbb) + {\rm Error} (m,\beta,N,\phi,k,\theta) 
  \end{multline}
  where $G_k$ is defined in \eqref{eq:G}. Additionally, there exists a constant $C(m,N) >0$ independent of $\phi$, $k$, and $\theta$ such that 
\begin{equation*}
   | {\rm Error} (m,\beta,N,\phi,k,\theta) | \leq  C(m,N)  \beta^{-1} \| \phi \|_{L^\infty (\Gamma_m)}. 
\end{equation*}
\end{lem}

\begin{proof}[Proof of \Cref{lem:Gibbs_precise_cap}] Using
  \Cref{lem:Gibbs_precise_cap_slice}, we have that
  \begin{align*}
   \frac1{\beta^{N}}   
         & \int_{ \SCap_k(\delta)} e^{-\beta (H(\bb)-h_*(m))} \phi(\bb)
    \Upsilon_{m}(\dbb) \\&=   \frac1{\beta^{N}}    \int_0^{2\pi}
         \int_{\SCap_k(\delta,\theta)} e^{-\beta (H(\bb)-h_*(m))} \phi(\bb)
    \Upsilon_{m,\theta}(\dbb) d\theta\\&=\frac1{\beta^{N}}     \int_0^{2\pi}
         \int_{\SCap_k(\delta,\theta)}  e^{-\beta G_k(\bb)}
    \Upsilon_{m,\theta}(\dbb) d\theta +   {\rm Error} (m,\beta,N,\phi) 
  \end{align*}
  where
  \begin{align*}
    {\rm Error} (m,\beta,N,\phi) &= \int_0^{2\pi}   {\rm Error}(m,\beta,N,\phi,k,\theta) d\theta\\
    & \leq  C(m,N)\frac{1}{\beta}\|\phi\|_{L^\infty(\Gamma_m)}.
  \end{align*}
 Here  ${\rm Error}(m,\beta,N,\phi,k,\theta)$ is the constant from
 \Cref{lem:Gibbs_precise_cap_slice}. The final bound come exactly from
 the estimate in that result.
\end{proof}

To complete the proof of \Cref{thm:Gibbs_precise}, we introduce the smooth partition of
unity, $\widetilde \psi,\psi_{k}\colon \C^{2N+1} \rightarrow [0,1]$ for
$k=-N,\dots,0\dots,N$ with the following properties.
Each $\tilde 
\psi,\psi_{k} \geq 0$ with 
$\psi_{k}$ is supported  in $\SCap_k(2\delta)$, 
$\psi_{k}(\bb)=1$ for all $\bb \in \SCap_k(\delta)$, and $\widetilde 
\psi(\bb)+\sum  \psi_k(\bb)=1$ for all $\bb \in \C^{2N+1}$. Notice 
that taking together these assumptions also imply that $\widetilde 
\psi(\bb)=0$ for all $\bb \in  \cup_k \SCap_k(\delta)$. 

\begin{proof}[Proof of \Cref{thm:Gibbs_precise}]
    \begin{multline}
    \frac1{\beta^{N}}   
         \int_{S(m)} e^{-\beta (H(\bb)-h_*(m))} \phi(\bb) \Gamma_m( \dbb)  
         \\
    =  \frac1{\beta^{N}}    \int_{S(m)} e^{-\beta (H(\bb)-h_*(m))}  \widetilde \psi(\bb)
    \phi(\bb) \Gamma_m( \dbb)  \\+\sum_k \frac1{\beta^N}   \int_{S(m)} e^{-\beta (H(\bb)-h_*(m))}   \psi_k(\bb)
    \phi(\bb) \Gamma_m( \dbb) .
  \end{multline}
  We will show that the first term on the right-hand side converges to
  zero as $\beta \rightarrow \infty$ and the remaining terms converge
  to the anticipated Gaussian integral.
\end{proof}

\subsection{A Convenient Choice of Local Coordinates and the proof of Lemma \ref{lem:Gibbs_precise_cap_slice}}
\label{sec:LocalCoordinates}

We begin by fixing any  $\bb \in \SCap_k(\delta) \in S(m)$ and recalling that
$\hat \bb_k^*(\bb)$ was defined,  around \eqref{eq:Gk},  to be the point in the set of 
minimizers $B_k^*=\{e^{i\theta}\bb_k: \theta \in [0,2\pi)\}$ that was
closest to $\bb$. Also recall that we always assume that $\delta$ was
sufficiently small to make $\hat \bb_k^*(\bb)$ unique. Also recall
that we set $\hat \theta_k^*(\bb)=\arg \hat \bb_k^*(\bb)$. Having
fixed $\bb$, we will simply write $\hat \bb_k^*$ and $\hat \theta_k^*$.

We now chose an orthonormal  basis $\{ \hat \bb_k^*, i  \hat \bb_k^*,
q_1,\dots,q_{4N}\}$ of $\C^{2N+1}$. Any $\bb' \in S(m)$ can be
represented uniquely as
\begin{align*}
  \bb = (1-t) \hat \bb_k^* + s i \hat \bb_k^* +\sqrt{2 t - t^2 -s^2} Q \bm{\Psi}
\end{align*}
where $Q$ is the matrix with columns $\{q_i: i =1, \dots,2N\}$ and
$\bm{\Psi} \in S^{2n}$, the unit ball in  $\R^{2n}$.
\begin{figure}[h]
    \centering 
    \begin{tikzpicture}
      \tkzInit[xmax=4,xmin=-4,ymax=3,ymin=-3]

      \tkzDefPoint(-2,0){lc}
      \tkzDefPoint(-3,1){ltl}
      \tkzDefPoint(-1,1){ltr}
      \tkzDefPoint(-2,0){ltm}
      \tkzDefPoint(-3,-1){lbl}
      \tkzDefPoint(-1,-1){lbr}
      \tkzDefPoint(-2,-1){lbm}

      \tkzDrawCircle[thick](lc,ltl); 
      \draw[thick] (-2,0) -- (-2,1.414); 
       \draw[thick] (-2,1) -- (-1,1); 

\draw[thin,dashed] (-3,2) -- (-2,1); 

\draw[thin,dashed] (0,1.5) -- (-1,1); 

\draw[thin,dashed] (-3,-2) -- (-2.15,-.125); 

\draw[thin,dashed] (-2,0) -- (-3,1); 

\draw[thin,dashed] (-2,0) -- (-1,1); 
      \node [anchor=east] at (-3,2) {$(1-t) \hat\bb^*$}; 
      \node [anchor=east] at (-3,-2) {$i \hat \bb^*$}; 
       \node [anchor=west] at (0,1.5) {$(1-t) \hat \bb^* + si \hat \bb^*$};
        \node [anchor=west] at (1,1) {$+ \sqrt{2 t - t^2-s^2} Q\bo$};
       
        \node [anchor=north] at (-2,2) {$\hat \bb^*$}; 

        \node [circle, fill=black, inner sep=1.5pt] at (-2,0) {}; 

        \node [circle, fill=black, inner sep=1.5pt] at (-2,1) (1pt) {}; 

            \node [circle, fill=black,inner sep=1.5pt ] at (-2.15,-.125) (1pt) {}; 

                      \node [circle, fill=black,inner sep=1.5pt ] at (-1,1) (1pt) {}; 
   \draw [dashed] (ltm) ellipse (1.414 and 0.125); 
      \draw [dashed] (ltm) ellipse (0.125 and 1.414); 

        \draw [dashed, thick, blue] (-2,1) ellipse (1 and 0.17); 
        \shade[ball color=blue!10!white,opacity=0.3] (-2.7,1.2) arc (-225:45:1 and .25) arc (45:135:1 and 1);

    \end{tikzpicture}
          \caption{A cartoon of the coordinate system on the sphere with the representative spherical cap shaded and bordered in blue.}
      \label{fig:sphere}
\end{figure}
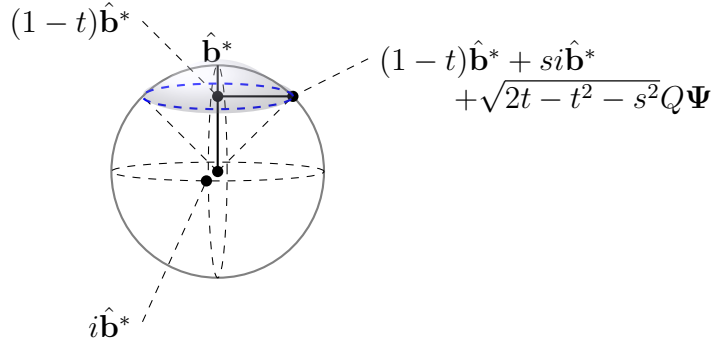

Notice that since  $\hat
\bb_k^*$ is the point in $B_k^*$ closest to $\bb$ and since $i \bb_k^*$ is
tangent to $B_k^*$ at the point $\bb_k^*$, we know that $s=0$.
Furthermore, observe  $\bb_k^*(\bb)$ and  $\hat \theta_k^*(\bb)$ are the
same for all $\bb \in  \SCap_k(\delta, \theta)$ (with  $\hat
\theta_k^*(\bb)=\theta$ in particular). Thus, for any $\bb \in
\SCap_k(\delta, \theta)$ we can write
\begin{align}\label{eq:genaric-cap-point}
  \bb & = (1-t) \hat \bb_k^* +\sqrt{2 t - t^2} Q \bm{\Psi} = \Proj_{\hat \bb_k^*} (\bb) + \Proj_{\hat \bb_k^*}^\perp (\bb) 
\end{align}
for some $t \in [0,1)$, $\bm{\Psi}  \in S^{2n}$ with a common $\hat \bb_k^*$ and an orthogonal matrix $Q$ corresponding to a choice of rotation that centers our coordinates around $\bb_k^*$ as the north pole of the sphere.  See Figure \ref{fig:sphere}.

\begin{proof}[Proof of \Cref{lem:Gibbs_precise_cap_slice}]

We begin by fixing a $\SCap_k(\delta,\theta)$. As already observed, for
all $\bb \in \SCap_k(\delta,\theta)$, $\widehat \bb_k^*(\bb)=\bb^*$ for some
common $\bb^*$ and $\arg \bb^* = \widehat \theta^*_k(\bb)=\theta$.

For $\bb \in \SCap_k(\delta,\theta)$, we wish to describe the structure of
\begin{align*}
    H(\bb) - H(\bb^*) & = H(\bb^*+\bxi) - H(\bb^*) \\
    & = G_k (\bb) + \tfrac16 \nabla^3 H (\bb^*) (\bxi,\bxi,\bxi) + H(\bxi).
\end{align*}

To make the integration more precise, we use the coordinate system introduced above in \Cref{fig:sphere}. Observe that any point $\bb \in \SCap_k(\delta,\theta)$, can be
written as $\bb = p \bb^* +q \bxi$ for some real $p$ and $q$, $p^2 + q^2 = 1$, and a $\bxi$
which is orthogonal to $\bb^*$. We now
expand $H(p\bb^*  + q\bxi ) - H(\bb^*)$ around the point $p \bb^*$ to
obtain 
\begin{align*}
H(p \bb^*  +q \bxi ) &=p^4H(\bb^*)+p^3 q \nabla H(\bb^*)  \bxi +\tfrac{ p^2q^2 }2 \nabla^2 H (\bb^*) (\bxi, \bxi) \\
& \hspace{2cm} +\tfrac{pq^3}6 \nabla^3 H (\bb^*) (\bxi, \bxi, \bxi) +\tfrac{q^4}{24} \nabla^4 H(\bxi, \bxi, \bxi, \bxi). 
\end{align*}
We note that since $H$ is a quartic polynomial, the Taylor series
expansion genuinely truncates after the $4$th order and as such we
have made no approximations in the above expression. This also
implies that $H(\bxi)=\tfrac{1}{24} \nabla^4 H(\bxi, \bxi, \bxi,
\bxi)$ a fact that we will use often in the sequel.

Using our observations about the  Lagrange 
multiplier from  \eqref{eqn:lagmult} and recalling that $\nabla H
(\bb^*) = -\tfrac{14}{11} m \bb^*$ and $| p \bb^* + q Q \bo|^2 =
|\bb^*|^2$, we see that 
\begin{equation}
  \label{eq:Hexp_pq}
  \begin{aligned}
p^3 q \nabla
H(\bb^*)  \bxi & = -\tfrac{7}{11} m p^2  [ (1-p^2) |\bb^*|^2 - q^2 |\bxi|^2 ] \\
& = \tfrac{7}{11}(p^2-1)p^2m^2+ \tfrac{p^2q^2}2
\tfrac{14}{11} m |\bxi|^2 \\
& =2 (1-p^2)p^2 H(\bb^*) + \tfrac{p^2q^2}2
\tfrac{14}{11} m |\bxi|^2 \notag
  \end{aligned}
\end{equation}
since $H(\bb^*) = -\tfrac{7}{22} m^2$.
Combining this with the above expansion yields
\begin{multline}\label{eq:H_xi-expansion}
  H(p \bb^*+ q\bxi)- H(\bb^*) =[p^4-1-2(p^2-1)p^2]H(\bb^*) \\+ \tfrac{ p^2q^2 }2
                               \langle ( \tfrac{14}{11} m I +
                                \nabla^2 H ) \bxi, \bxi\rangle +  R(p,q,\bxi)
\end{multline}
where $R(p,q,\bxi)= \tfrac{pq^3}6 \nabla^3 H (\bb^*) (\bxi, \bxi,
\bxi) +q^4H(\bxi)$.

It is convenient to fix $m=1$ from here onward in the proof without loss of generality and consider
the coordinate representation on the sphere given by 
\begin{equation}
\label{eqn:capcoords}
\bb(t,\bo)\eqdef \bb^* + \bxi =
(1-t)  \bb^* + \sqrt{2 t -  t^2} Q \bo
\end{equation}
from \eqref{eq:genaric-cap-point}, i.e. $\sqrt{2 t -  t^2} Q \bo $. Since $\bb^*$ orthogonal to $Q
\bo$, we can use the \eqref{eq:H_xi-expansion} with $p=1-t$ and $q=
\sqrt{2 t -  t^2}$.

Then, we write 
\begin{equation}
  \label{Hshpexp_alt}
    \begin{aligned}
   & H(b(t,\bo))-H(\bb^*)  = H((1-t) \bb^* + \sqrt{2 t - t^2} Q \bo ) - H(\bb^*) \\
 =& - t^2 (2-t)^2 H(\bb^*)  +  \frac12 (1-t)^2 (2t - t^2) \left \langle \left( \tfrac{14}{11} I + \nabla^2 H(\bb^*) \right) Q \bo, Q \bo \right \rangle \notag  \\
 &+ (1-t) (2 t - t^2)^{\frac32}  \tfrac16 \nabla^3 H(\bb^*)
 (Q \bo, Q \bo,Q \bo) + (2t - t^2)^2 H(Q
 \bo) . \notag
  \end{aligned}
\end{equation}

We thus define 
\begin{align}
    \label{Gfunc}
    \G_k(t,\bpsi) &  \eqdef \tfrac12  (1-t)^2 (2t - t^2) \left \langle \left( \tfrac{14}{11} I + \nabla^2 H(\bb^*) \right) Q \bo, Q \bo \right \rangle  
\end{align}
so
\begin{equation*}
   H(b(t,\bo))-H(\bb^*)  = - t^2 (2-t)^2 H(\bb^*) + \G_k(t,\bpsi) + \K_k(t,\bo)
\end{equation*}
for
\begin{align*}
    \K_k(t,\bpsi) & \eqdef (1-t) (2 t - t^2)^{\frac32}  \tfrac16 \nabla^3 H(\bb^*) (Q \bo, Q \bo,Q \bo) \\ 
    & + (2t - t^2)^2 H(Q \bo).
\end{align*}

Using the identity \eqref{eqn:H2bprod}, the definition of $G_k$ in \eqref{eq:Gk}, and $\bxi (\bb) = -t \bb^* + \sqrt{2t - t^2}Q \bo$, we observe that
\begin{equation}
\label{eqn:GvsGk}
\begin{array}{ll}
G_k (\bb) & \!\!\!\! =\tfrac12  \left \langle \left( \tfrac{14}{11} I + \nabla^2 H(\bb^*) \right) (I - \langle \bb^*, \bxi \rangle) \bxi (\bb), (I - \langle \bb^*, \bxi  \rangle) \bxi(\bb) \right \rangle \\
& \!\!\!\!  = \G_k(t,\bo) + (1 - (1-t^2)) \G_k(t,\bo)   \\
& \!\!\!\!  = \G_k(t,\bo) + (2t-t^2) \G_k(t,\bo), 
\end{array}
\end{equation}
meaning that $G_k(\bb)$ and $\G_k(t,\bo)$ are the same at $t=0$, where we
will see the measure concentrates, and they are small perturbations of
each other for $0< t < \delta$ when $\delta$ is small.

With these preliminaries out of the way, we turn to the proof the fact
that the integral in \Cref{lem:Gibbs_precise_cap_slice} is well
approximated by a Gaussian-like measure as $\beta \rightarrow
\infty$. We recall that by assumption the smooth test function $\phi$
is supported inside of $\SCap(\delta,\theta)$.

We write
\begin{multline*} 
         \int_{\SCap_k(\delta)} e^{-\beta (H(\bb)-h_*(m))} \phi(\bb)
  \Gamma_m( \dbb) = \int_0^{\delta} (2 t-t^2)^{2 N} \\\times \int_{S^{4N}}
  e^{ \beta t^2 (2-t)^2 h_*(m)} e^{-\beta[\G_k(t,\bpsi) + \K_k(t,\bpsi)]} \phi(b(t,\bo))
  \sigma (d\bo) d t 
\end{multline*}
where $b(t,\bo)=(1-t)  \bb^* + \sqrt{2 t -  t^2} Q \bo$.
Take
\begin{multline*}
\int_0^{\delta} (2 t-t^2)^{2 N}  \int_{S^{4N}}
 e^{ \beta t^2 (2-t)^2 h_*(m)}  e^{-\beta[\G_k(t,\bpsi) + \K_k(t,\bpsi)]} \phi(b(t,\bo))
  \sigma (d\bo) d t \\ = 
 \int_0^{\delta} (2 t-t^2)^{2 N} e^{ \beta t^2 (2-t)^2 h_*(m)} \int_{S^{4N}} e^{-\beta \G_k } \phi(b(t,\bo))
  \sigma (d\bo) d t \\
\ \ +  \int_0^{\delta} (2 t-t^2)^{2 N}  e^{ \beta t^2 (2-t)^2 h_*(m)} \int_{S^{4N}} e^{-\beta \G_k } [ e^{-\beta K} - 1] \phi(b(t,\bo))
  \sigma (d\bo) d t  .
\end{multline*}
Note, using \eqref{eqn:GvsGk}, we see that the function $\G_k$ above can be replaced by $G_k$.

Thus, we will focus our attention on bounding 
\begin{equation*}
\int_0^{\delta} (2 t- t^2)^{2 N} e^{ \beta t^2 (2-t)^2 h_*(m)} \int_{S^{4N}} e^{-\beta \G_k } [ e^{-\beta K} - 1] \phi(b(t,\bo))
  \sigma (d\bo) d t 
\end{equation*}
as $\beta \to \infty$.  
For a given $\beta$, rescaling $\omega \to \sqrt{\beta}^{-1} \omega$, we get
\begin{equation*}
 \beta^{2N} \int_0^{\delta} (2 t-t^2)^{2 N} e^{ \beta t^2 (2-t)^2 h_*(m)} \int_{S^{4N}} e^{-\G_k} [ e^{-\tilde K} - 1]  \phi(b(t,\bo))
  \sigma (d\bo) d t  ,
\end{equation*}
where
\begin{multline*}
    \widetilde \K_k(t,\bo)  \eqdef \frac{1}{ \sqrt{\beta}} (1-t) (2 t - t^2)^{\frac32}  \frac16 \nabla^3 H(\bb^*) (Q \bo , Q \bo,Q \bo) \\
    + \frac{1}{\beta}  (2t - t^2)^2 H(Q \bo).
\end{multline*}
Using such a decomposition, we are now easily able to integrate out all the terms to prove Theorem
\ref{thm:Gibbs_precise} and hence Theorem \ref{thm:Gibbs}.  Indeed, we observe that for $\beta$ sufficiently large and $\delta < 3/4$, we have
\begin{align*}
& \beta^{2N} \int_0^{\delta} (2 t-t^2)^{2 N} e^{ \beta t^2 (2-t)^2 h_*(m)} \int_{S^{4N}} e^{-G} [ e^{-\tilde K} - 1]  \phi(b(t,\bo))
  \sigma (d\bo) d t  \\
& \hspace{1cm}  \leq C(m,N)  \beta^{2N-1/2} \int_0^\delta u^{2N} e^{- \beta m \frac{7}{22} u^2} du \leq C(m,N)  \beta^{N-1} ,
\end{align*}
where we have abused notation and updated the $\beta$-independent constant $C(m,N)$ from line to line.

Thus, the smallness of $e^{-\tilde K } -1$ means this bound is thus lower order in $\beta$ from that of
the integral 
\begin{equation*}
\int_0^{\delta} (2 t- t^2)^{2 N} e^{ \beta t^2 (2-t)^2 h_*(m)} \int_{S^{4N}} e^{-\beta \G_k }  \phi(b(t,\bo))
  \sigma (d\bo) d t .
\end{equation*}

It is also now clear that the contribution of the measure is exponentially small as $\beta \to \infty$ when considering any contribution away from $t = 0$.  In other words, we converge precisely to a measure of the form described in Theorem \ref{thm:Gibbs}.  To see this, we have observed that
\begin{multline*} 
         \int_{\SCap_k(\delta)} e^{-\beta (H(\bb)-h_*(m))} \phi(\bb)
  \Gamma_m( \dbb) \approx \\
  \int_0^{\delta} (2 t-t^2)^{2 N} \times \int_{S^{4N}}
  e^{ \beta t^2 (2-t)^2 h_*(m)} e^{-\beta[\G_k(t,\bpsi)]} \phi(b(t,\bo))
  \sigma (d\bo) d t .
\end{multline*}
On the set of optimizers, we have that $e^{-\beta (H(\bb)-h_*)}  = 1$, $G(i \bb^*) = 0$ and recall \eqref{eqn:GopBound}.  
We thus observe that outside a small collar around the neighborhood of optimizers of size $\beta^{-s}$ for $\tfrac12 < s < 1$, the operator $G$ ensures exponentially small contributions since on $\SCap_k(\delta)$, we have 
\begin{equation*}
e^{-\beta[\G_k(t,\bpsi)]} \leq e^{-\beta^{1-s} (2t-t^2) \tfrac{2}{11} |\bpsi|^2} \leq e^{-\tfrac{3}{11} \beta^{1-s} }
\end{equation*}
since $2 \beta^{1-s} - \beta^{1-2s} > \tfrac32 \beta^{1-s}$ for $\beta$ sufficiently large.

We can thus compare this expansion with the integral arising in \eqref{eq:gaussian}, namely
\begin{multline*} 
\int_{\SCap_k(\delta)} e^{-\beta (H(\bb)-h_*(m))} \phi(\bb)
  \Gamma_m( \dbb)  \\
  \approx
\int_0^{\delta} (2 t-t^2)^{2 N} \times \int_{S^{4N}}
  e^{ \beta t^2 (2-t)^2 h_*(m)} e^{-\beta[\G_k(t,\bpsi)]} \phi(b(t,\bo))
  \sigma (d\bo) d t  \\
  \approx \int_0^{\delta} (2 t-t^2)^{2 N}  \times \int_{S^{4N}}
  e^{ \beta t^2 (2-t)^2 h_*(m) } e^{-\beta \left[\frac{G_k(t,\bpsi)}{(1+2t-t^2)} \right] } \phi(b(t,\bo))
  \sigma (d\bo) d t \\
 \\
 \approx \int_0^{\delta} (2 t-t^2)^{2 N} \times \int_{S^{4N}}
e^{-\beta G_k(b(t,\bpsi))} \phi(b(t,\bo))
  \sigma (d\bo) d t  \\
= \int_{\SCap_k(\delta)} e^{-\beta G_k(\bb)} \phi(\bb)
  \Gamma_m( \dbb) .
\end{multline*} 
Above, we have used the identity\eqref{eqn:GvsGk} as well as the coordinate system \eqref{eqn:capcoords} for the spherical cap.

\end{proof}

We now sketch the proofs of \Cref{rem:contration} and
\Cref{c:GareTheSame}. The proofs are relatively standard but follows as
byproducts of some of the computations in the proof of
\Cref{lem:Gibbs_precise_cap_slice} given above.

\begin{proof}[Proof of  \Cref{rem:contration} and
  \Cref{c:GareTheSame}]
Looking back at the proof of \Cref{lem:Gibbs_precise_cap_slice}, it
was shown that as $\beta \rightarrow \infty$, the measures all
concentrate in the set $\bigcup_k \SCap_k(\delta)$ for any $\delta
>0$. Since  for $\epsilon >0$,  $\bigcup_k \SCap_k(\delta) \subset
B_k^{*\epsilon}$ for $\delta$ small enough the results in
\Cref{rem:contration} follow.

Similarly since
\begin{align*}
  \lim_{\epsilon \rightarrow 0}\sup_{ \bb \in  B^{*\epsilon}}|
  G(\bb) -  \widetilde G(\bb)| =\lim_{\epsilon \rightarrow 0}  \sup_k \sup_{ \bb \in  B_k^{*\epsilon}}|
  G(\bb) -  G_k(\bb)| = 0
\end{align*}
  the results of  \Cref{c:GareTheSame} follow from the concentration
  results of  \Cref{rem:contration}.

\end{proof}

\section{Discussion and Extensions}
\label{discussion}

From the characterization of the Gibbs measure, we can easily observe that solutions have support on every lattice node with probability $1$.  In this setting, using \cite{CKSTT}, that correlates to high frequencies having non-trivial support for generic solutions in a neighborhood of the minimizing solutions

A similar compact stationary minimizing solution in Theorem \ref{thm:3ModeRho} was found for a continuum version of this problem in \cite{germain2017compactons}.  The authors also described traveling compact solutions that are infinite energy.  Existence theory in a neighborhood of such a compact supported solution has only begun to be studied in \cite{germain2019existence,harropLWP}.  While there are notions of Mass conservation and a similar Hamiltonian, the invariance of the measure is much more challenging given the difficulty is studying the evolution equations.  However, a hydrodynamics formulation exists and gives further insight into the expected dynamics.  Developing a related Gibbs measure framework to that constructed here for the continuum model in \cite{germain2017compactons} is an interesting future direction for study.

For related discrete models, the recent paper \cite{parker2024standing} studied stationary solutions for a generalization of our model of the form
\begin{equation}
  \label{e:toy_model_gen}
  -i\dt b_j (t) = -\abs{b_j(t)}^2 b_j(t) + D ( b_{j-1}^2 \overline{b_j}(t)
  + b_{j+1}^2 \overline{b_j}(t)) ,
\end{equation}
where $D$ is now a parameter (for us, $D=2$).  For $D = \frac12$, the authors find that a $10$-mode state is the minimizer of the energy for instance. We expect some of our results here to apply to this more general setting, though some of the explicit rearrangement arguments in Section \ref{Sec:minimizer} would need to be appropriately modified.

\bibliographystyle{abbrv}

\bibliography{ToyModel2}

\end{document}